\documentclass[11pt, a4paper]{article}
\usepackage[utf8]{inputenc}
\usepackage{mathtools}
\usepackage{amsmath}
\usepackage{amsthm}
\usepackage{algorithm}
\usepackage{algorithmic}
\usepackage{amssymb}
\usepackage{epsfig}
\usepackage{titlesec}
\allowdisplaybreaks

\usepackage{stackengine}
\mathtoolsset{showonlyrefs}
\usepackage{xcolor}
\usepackage[a4paper, margin=2.5cm]{geometry}
\usepackage{comment}
\usepackage{mathrsfs} 
\usepackage[export]{adjustbox}

\usepackage{multirow}

\usepackage{dsfont}
\usepackage{fixltx2e}
\usepackage{float}
\usepackage{mathrsfs}
\usepackage{amssymb,amsmath,amsthm,amsfonts}
\usepackage{mathtools}
\usepackage{fixmath}
\usepackage{bbm}
\usepackage{braket}
\usepackage{caption}
\usepackage{subcaption}
\usepackage{enumitem}
\usepackage{units}
\usepackage{xcolor}
\usepackage{booktabs}
\usepackage[english]{babel}
\usepackage[T1]{fontenc}
\usepackage{lmodern}
\usepackage{tikz}
\usepackage{pgfplots}
\usepackage{algorithm}
\usepackage{algorithmic}
\usepackage{comment}

\usepackage[final,stretch=10]{microtype}
\usepackage[colorlinks, citecolor=hanblue,linkcolor=black]{hyperref}
\definecolor{hanblue}{rgb}{0.27, 0.42, 0.81}
\usepackage{todonotes}

\newcommand{\N}{\mathbb{N}}

\newcommand{\R}{\mathbb{R}}

\newcommand{\Ext}{\operatorname{Ext}}

\newcommand{\RR}{\mathbb{R}}      % for Real numbers
\newcommand{\tr}[1]{tr(C(1)^{-1})}

\newcommand{\vertiii}[1]{{\left\vert\kern-0.25ex\left\vert\kern-0.25ex\left\vert #1
		\right\vert\kern-0.25ex\right\vert\kern-0.25ex\right\vert}}

\newcommand{\ds}{\displaystyle}

\newcommand{\supp}{{\rm Supp} (\mu_0)}

\newcommand{\dm}{\mathrm{d}}

\renewcommand{\leq}{\leqslant}
\renewcommand{\geq}{\geqslant}

\makeatletter
\DeclareFontFamily{OMX}{MnSymbolE}{}
\DeclareSymbolFont{MnLargeSymbols}{OMX}{MnSymbolE}{m}{n}
\SetSymbolFont{MnLargeSymbols}{bold}{OMX}{MnSymbolE}{b}{n}
\DeclareFontShape{OMX}{MnSymbolE}{m}{n}{
	<-6>  MnSymbolE5
	<6-7>  MnSymbolE6
	<7-8>  MnSymbolE7
	<8-9>  MnSymbolE8
	<9-10> MnSymbolE9
	<10-12> MnSymbolE10
	<12->   MnSymbolE12
}{}
\DeclareFontShape{OMX}{MnSymbolE}{b}{n}{
	<-6>  MnSymbolE-Bold5
	<6-7>  MnSymbolE-Bold6
	<7-8>  MnSymbolE-Bold7
	<8-9>  MnSymbolE-Bold8
	<9-10> MnSymbolE-Bold9
	<10-12> MnSymbolE-Bold10
	<12->   MnSymbolE-Bold12
}{}

\let\llangle\@undefined
\let\rrangle\@undefined
\DeclareMathDelimiter{\llangle}{\mathopen}%
{MnLargeSymbols}{'164}{MnLargeSymbols}{'164}
\DeclareMathDelimiter{\rrangle}{\mathclose}%
{MnLargeSymbols}{'171}{MnLargeSymbols}{'171}
\makeatother

\makeatletter
\newcommand*\rel@kern[1]{\kern#1\dimexpr\macc@kerna}
\newcommand*\widebar[1]{%
	\begingroup
	\def\mathaccent##1##2{%
		\rel@kern{0.8}%
		\overline{\rel@kern{-0.8}\macc@nucleus\rel@kern{0.2}}%
		\rel@kern{-0.2}%
	}%
	\macc@depth\@ne
	\let\math@bgroup\@empty \let\math@egroup\macc@set@skewchar
	\mathsurround\z@ \frozen@everymath{\mathgroup\macc@group\relax}%
	\macc@set@skewchar\relax
	\let\mathaccentV\macc@nested@a
	\macc@nested@a\relax111{#1}%
	\endgroup
}
\makeatother

\numberwithin{equation}{section}

\definecolor{darkred}{rgb}{.7,0,0}

\definecolor{green}{rgb}{0,0.7,0}

\usepackage[notref,notcite]{}
\theoremstyle{plain}

\newtheorem{theorem}{Theorem}[section]
\newtheorem{Lemma}[theorem]{Lemma}

\newtheorem{proposition}[theorem]{Proposition}

\theoremstyle{definition} 

\newtheorem{assumption}[theorem]{Assumption}
\newtheorem{definition}[theorem]{Definition}

\newtheorem{remark}[theorem]{Remark}

\usepackage{scalerel}

\title{A Dual Certificate Approach to Sparsity in Infinite-Width Shallow Neural Networks 
}
\author{Leonardo Del Grande \!\!\thanks{Department of Applied Mathematics, University of Twente, 7500AE Enschede, The Netherlands \\
(\texttt{l.delgrande{@}utwente.nl}, \texttt{c.brune{@}utwente.nl}, \texttt{m.c.carioni@utwente.nl})} , Christoph Brune\footnotemark[1] , Marcello Carioni\footnotemark[1]
}
\date{}

\begin{document}
 
	%\title{Exact Sparse Reconstruction In General Banach Spaces Using a Metric Non-Degenerate Source Condition}
	
	%\pagestyle{myheadings}
	
	%\author[M. Carioni]
	%{Marcello Carioni}
	%\address[Marcello Carioni]{University of Twente, Department of Mathematics for Imaging and AI, Zilverling, 7500 AE Enschede, NL
		%}
	%\email{m.c.carioni@utwente.nl}

 %\author[L. Del Grande]
%	{Leonardo Del Grande}
%	\address[Leonardo Del Grande]{% the University of Twente, Department of Mathematics for Imaging and AI, Zilverling, 7500 AE Enschede, NL
		%}
	%\email{l.delgrande@utwente.nl}
	\maketitle

\begin{abstract}
\noindent 
In this paper, we study total variation (TV)-regularized training of infinite-width shallow ReLU neural networks, formulated as a convex optimization problem over measures on the unit sphere. Our approach leverages the duality theory of TV-regularized optimization problems to establish rigorous guarantees on the sparsity of the solutions to the training problem. Our analysis further characterizes how and when this sparsity persists in a low noise regime and for small regularization parameter.
The key observation that motivates our analysis is that, for ReLU activations, the associated dual certificate is piecewise linear in the weight space. Its linearity regions, which we name dual regions, are determined by the activation patterns of the data via the induced hyperplane arrangement. Taking advantage of this structure, we prove that, on each dual region, the dual certificate admits at most one extreme value. As a consequence, the support of any minimizer is finite, and its cardinality can be bounded from above by a constant depending only on the geometry of the data-induced hyperplane arrangement. Then, we further investigate sufficient conditions ensuring uniqueness of such sparse solution. Finally, under a suitable non-degeneracy condition on the dual certificate along the boundaries of the dual regions, we prove that in the presence of low label noise and for small regularization parameter, solutions to the training problem remain sparse with the same number of Dirac deltas. Additionally, their location and the amplitudes converge, and, in case the locations lie in the interior of a dual region, the convergence happens with a rate that depends linearly on the noise and the regularization parameter.
\end{abstract}
\section{Introduction}

Two-layer neural networks with $K$ \emph{neurons} are parametrized functions $F_{c,w,b}:\mathbb{R}^d\rightarrow \RR$ defined as 
\begin{equation}\label{2NN}
    F_{c,w,b}(x)=\sum_{i=1}^K c^i \sigma(\langle w^i, x\rangle + b^i),
\end{equation}
where $\sigma : \mathbb{R}\rightarrow \mathbb{R}$ is the so-called \emph{activation function}, while $c^i \in \mathbb{R}$ and $w^i \in \mathbb{R}^d$ are the weights for the output and hidden layers of the network, and $b^i \in \mathbb{R}$ are the biases. Following the approach in \cite{bach2017breaking, weinan2022representation} one can consider the formal limit for $K\rightarrow +\infty$, i.e. when the number of neurons tends to infinity and replace the sum in \eqref{2NN} by the integration over a measure taking values on the space of weights and biases. In particular, one defines
\begin{align}\label{eq:infinitewidth}
    F_{\mu}(x) = \int_{\Theta} \sigma(\langle w,x\rangle + b) \, d\mu(w,b)
\end{align}
where $(w,b) \in \Theta$ and $\mu \in M(\Theta)$. Problem \eqref{eq:infinitewidth} extends \eqref{2NN} since by choosing measures $\mu$ of the form $\mu = \sum_{i=1}^N c^i\delta_{(w^i,b^i)}$ one recovers instances of \eqref{2NN}. Moreover, despite being infinite-dimensional it is convex in the space of measures \eqref{eq:infinitewidth}, allowing for insights facilitated by the tools of convex analysis (see for instance \cite{pikka}). As often considered in the available literature, the bias $b$ is reabsorbed into the weight $w$ by identifying $x = (x,1)$ and $w = (w,b)$, allowing for a more compact representation 
\begin{align}\label{eq:infinitewidth2}
    F_{\mu}(x) = \int_{\Theta} \sigma(\langle w,x\rangle) \, d\mu(w).
\end{align}
In this work, we consider positively $1$-homogeneous activation functions $\sigma:\R\rightarrow \R$ and, in particular, we often restrict our attention to the ReLU activation function defined as $\sigma(z) = \max\{0,z\}$ for $z \in \R$. Due to such $1$-homogeneity, since $\sigma(\langle w,x\rangle) =  |w|\sigma(\langle w/|w|,x\rangle)$ and rescaling the measure $\mu$ one can, without loss of expressivity for $F_\mu$, consider measures taking values on the sphere, that is $\Theta = \mathbb{S}^{d-1}$.  Such a choice is also convenient, since the compactness of the parameter space allows for well-posed variational problems in the space of measures. We will make the choice $\Theta= \mathbb{S}^{d-1}$ throughout this paper.

According to the previous considerations, given a collection of input/output pairs $(x^j,y^j)_{j=1}^n$ for $n\in \mathbb{N}$, a typical training problem taking values in the space of measures $M(\mathbb{S}^{d-1})$ can be defined as 
\begin{align}\label{eq:reg_intro}
    \min_{\mu \in M(\mathbb{S}^{d-1})} \frac{1}{2}\sum_{j=1}^n |F_{\mu}(x^j) - y^j|^2 + \lambda \|\mu\|_{\mathrm{TV}}, 
\end{align}
where $\|\mu\|_{\mathrm{TV}}$ is
the total variation norm, defined as 
\begin{align}\label{eq:tv}
\|\mu\|_{\mathrm{TV}} = \sup \left\{ \int \varphi(x)\,d\mu(x) : \varphi \in C(\mathbb{S}^{d-1})\text{ such that } \|\varphi\|_{\infty} \leq 1\right\}.
\end{align}
In this case, the total variation norm acts as a regularizer (whose intensity is regulated by $\lambda$). In particular, it has the  
goal of enforcing sparsity in the solution. This means selecting solutions that are made of a linear combination of Dirac deltas, thus recovering a classical neural network (with a finite number of neurons).

In the regime $\lambda \ll 1$, problem \eqref{eq:reg_intro} approaches an interpolation problem, in which the data fidelity term effectively acts as a hard constraint. Formally, one recovers the constrained formulation 
\begin{align}\label{eq:reg_hard_intro}
    \min_{\mu \in M(\mathbb{S}^{d-1})} \|\mu\|_{\mathrm{TV}} \quad\quad  \text{subjected to } \quad F_{\mu}(x^j) = y^j \ \ \ \forall j.
\end{align}
We also remark that total variation norm regularization allows to define a natural representation costs for infinite-width shallow neural networks 
opening the way to the analysis of approximation capabilities of shallow neural networks.
Indeed, for $f \in L^1(\mathbb{S}^{d-1})$ one can define   
\begin{align}
    \|f\|_{\mathcal{B}(\mathbb{S}^{d-1})} = \inf\left\{\|\mu\|_{\mathrm{TV}} : \mu \in M(\mathbb{S}^{d-1}),   F_\mu(x) = f(x) \text{ for } a.e. \ x\in \mathbb{S}^{d-1} \right\},
\end{align}
that can be seen as a representation cost in infinite-width regimes. 
In analogy with foundational works by Barron \cite{barron1994approximation, barron2002universal},
the quantity $\|f\|_{\mathcal{B}(\mathbb{S}^{d-1})}$ is also named Barron norm and 
Barron functions are those functions $f \in L^1(\mathbb{S}^{d-1})$ such that $\|f\|_{\mathcal{B}(\mathbb{S}^{d-1})} < \infty$ (see for instance \cite{ma2022barron}). 
Many important works have carried out different aspects of such analysis related either to the study of the representation cost or the approximation power of Barron functions \cite{dummer2025vector, ma2022barron, weinan2022representation, weinan2022some, heeringa2024embeddings, ongie2019function, savarese2019infinite, shevchenko2022mean}.

\subsection{Sparse guarantees in infinite-width limit of neural networks: state of the art and super-resolution theory}\label{sec:state-of-art}

It is a natural question to ask under which assumptions solutions of problem \eqref{eq:reg_intro} and \eqref{eq:reg_hard_intro} are sparse, namely their minimizers are made of linear combinations of Dirac deltas of the form
\begin{align}\label{eq:diracs_intro}
    \mu = \sum_{i=1}^N c^i\delta_{w^i}.
\end{align}
As highlighted in \cite{bartolucci2023understanding} in the context of Reproducing Kernel Banach Spaces (RKBS), a first, partial answer is provided by so-called representer theorems \cite{boyer, bredies2020sparsity}. In particular, when choosing the regularizer $R(\mu)=\|\mu\|_{TV}$ and assuming that the weight space $\Theta$ is compact, general representer theorems for infinite-dimensional convex optimization problems guarantee the existence of at least one solution of the form \eqref{eq:diracs_intro}, with the number of atoms bounded by $N\leq n$. Such sparsity results have been established in various training frameworks \cite{bartolucci2024lipschitz, bartolucci2023understanding, parhi2021banach, parhi2022kinds}. 
However, empirical evidence suggests that sparsity is not merely an existence property: in practice, all minimizers often appear to be sparse even when multiple solutions exist. This naturally raises the question of whether all solutions are sparse.
While this question remains largely unexplored in the machine learning literature with few exceptions \cite{dedios2020sparsity, najaf2025sampling}, it has been studied in depth in the context of super-resolution, where similar problems are formulated as convex optimization problems over spaces of measures. In this context, the property of having unique and sparse solutions is typically referred to as \emph{exact reconstruction}. Foundational results \cite{CandSR, de2012exact,schiebinger2018superresolution} identify precise conditions under which exact reconstruction holds for TV-regularized interpolation problems with Fourier constraints. %Foundational results \cite{CandSR,de2012exact,schiebinger2018superresolution} provide precise conditions for exact reconstruction in TV-regularized interpolation problems with Fourier constraints.
These results typically rely on duality arguments, where the key tool is the construction of dual certificates $\eta$, i.e. dual variables that certify optimality and sparsity of the primal solution by satisfying the associated first order optimality conditions \cite{fuchs2004sparse}. 

Beyond exact recovery, a substantial body of work has focused on sparse stability, namely the robustness of sparsity under perturbations of the problem. In this setting, one seeks conditions ensuring that the sparsity  of the solution persists when the regularization parameter 
$\lambda$ varies, and the data and labels are contaminated by noise.
A series of influential works \cite{denoyelle2017support, duval2020characterization, Peyre,  offthegridPoon}  has shown that sparse stability is governed by non-degeneracy properties of the dual certificate associated with the optimization problem. More precisely, stability holds when the dual certificate saturates its extreme values on the support of the underlying sparse measure and exhibits a strict second order behavior there, namely local strong concavity (or convexity, depending on the sign convention). These conditions prevent the creation or annihilation of atoms under small perturbations. %Instead, with ReLU features the dual certificate is piecewise linear in the parameter space, allowing a region-by-region geometric analysis, and leading to stability results that do not require second-order non-degeneracy.

\subsection{Contributions of the paper}

In this paper, we study sparsity properties of both problems \eqref{eq:reg_intro} and \eqref{eq:reg_hard_intro} in the case of the ReLU activation function. Our analysis follows a duality-based approach, drawing on classical tools from optimization in the space of measures and super-resolution techniques discussed in the previous part of the introduction. Duality methods play a central role in infinite-dimensional optimization and in the analysis of sparse representations. Nevertheless, their systematic use to establish sparsity guarantees for infinite-width neural networks remains largely unexplored. We note, however, that duality techniques have appeared in the context of Barron spaces \cite{spek2025duality} and in certain finite-dimensional settings \cite{kim2024exploring}.

We begin our analysis by proving that every solution of both problems is sparse, and that the number of Dirac deltas in any minimizer is uniformly bounded from above by a quantity depending only on the ambient dimension $d$ and on the number of data points $n$.
The key observation underlying our results is that the dual certificate associated with the problems can be written in the explicit form
 \begin{align}\label{intro:sparse}
        \eta(w)= \sum_{i=1}^n p^i \sigma (\langle w, x^i \rangle), \quad p^i \in \mathbb{R}, w\in \mathbb{S}^{d-1}. 
 \end{align}
When $\sigma$ is the ReLU activation function, $\eta$ coincides with the restriction to the sphere $\mathbb{S}^{d-1}$ of a piecewise linear function. More precisely, its linearity regions are determined by the intersection of the half-spaces $\{w: \langle w, x^i \rangle >0\}$ and $\{w: \langle w, x^i \rangle <0\}$, referred to as \emph{dual regions}, and $\eta$ changes slope precisely across their common boundaries $\left\{w: \langle w, x^j\rangle=0\right\}$.

Due to the geometry of the sphere, this observation allows us to conclude that in the interior of each dual region, $\eta$ achieves maximum and minimum in at most one point. Moreover, at the common boundary of the regions (determined by the vanishing of the scalar products $\langle w, x^i\rangle$), a similar analysis can also be performed.  Since, by optimality conditions, the extreme values of the dual certificates determine the support of the minimizers, one can then infer the sparsity of the minimizers. Additionally, by combinatorial arguments involving the counting of non-empty dual regions, one can also obtain an upper bound on the number of Dirac deltas constituting the minimizers. 

It is important to emphasize that the sparsity of minimizers alone does not imply uniqueness. In Section~\ref{uniq}, we introduce sufficient conditions guaranteeing the uniqueness of the minimizer. These conditions rely on two key properties: first, given a sparse measure of the form \eqref{eq:diracs_intro}, the associated dual certificate for problems~\eqref{eq:reg_intro} and~\eqref{eq:reg_hard_intro} attains the extreme values $|\eta(w)| = 1$ exactly at the support points $w^i$; second, the corresponding neural network evaluations at the data points, namely ${\int_{\mathbb{S}^{d-1}} \sigma(\langle w, x^j\rangle)\mathrm{d}\mu(w)}$, are linearly independent. In this section we provide sufficient conditions that guarantee such linear independence.

Finally, Section~\ref{sec:stability} is devoted to the analysis of sparse stability for problem~\eqref{eq:reg_hard_intro}. More precisely, we show that, under suitable assumptions on the dual certificate and for sufficiently small regularization parameters and noise levels in the labels, the sparsity of the solution is preserved, with the same number of Dirac deltas. Moreover, the weights $c^i$ and $w^i$ of the minimizer of the perturbed problem converge to the weights of the unperturbed one.
As discussed in Section~\ref{sec:state-of-art}, a key ingredient for sparse stability is the non-degeneracy of the dual certificate associated with problem~\eqref{eq:reg_hard_intro}. While this remains true in our setting, the piecewise linear structure of the dual certificate leads to non-degeneracy conditions with an explicit characterization. In particular, when a Dirac delta lies in the interior of a dual region, no additional non-degeneracy condition is required. In contrast, when a Dirac delta is located at the boundary of a region where the dual certificate is non-differentiable, the non-degeneracy is characterized by sign conditions on the (distributional) derivative of the dual certificate.

Finally, in the case where the support points $w^i$ lie in the interior of a dual region, we establish quantitative convergence rates for the weights of the minimizer of the perturbed problem. More precisely, under suitable linear independence assumptions on both the neural network evaluations at the data points and their derivatives, we prove that

\begin{equation}
    \begin{aligned}
\left|c_{\lambda,\zeta}^i-c_0^i\right|&=O(\lambda), \\
\left\|w_{\lambda,\zeta}^i-w_0^i\right\|&=O(\lambda),
    \end{aligned}
\end{equation}

as $\lambda \to 0$ and the noise $\zeta \to 0$ with $\|\zeta\| \leq \alpha \lambda$.

\section{Notations and preliminaries}

\subsection{TV-regularized optimization in the space of measures}\label{sec:tvreg}
In this section, we review tools of optimization in the space of measures we are going to use in the rest of the paper. We denote by $M(X)$ the Banach space of finite signed Radon measures on a compact set $X \subset \mathbb{R}^d$, endowed with the total variation norm.
For any Radon measure $\mu$ defined on $X$, we  also denote its support by $\operatorname{Supp}(\mu)$.
We will consider TV-regularized optimization problems in the space of measures (commonly known as BLASSO \cite{azais2015spike}) that will be written as 
\begin{align}\label{eq:softproblem}
    \inf_{\mu \in M(X)} \frac{1}{2}\|K\mu - y_\zeta\|_Y^2 + \lambda \|\mu\|_{\rm TV},  
    \tag{$\mathcal{P}_\lambda(y_\zeta)$}
\end{align}
where $Y$ is a separable Hilbert space, $y_\zeta\in Y$, $K: M(X) \rightarrow Y$ is a weak*-to-strong continuous linear operator and $\lambda$ is a positive parameter. Such problems are popular as inverse problems in the space of measures with applications to signal processing \cite{CandSR}, super-resolution in microscopy \cite{Denoyelle_2019}, sensor placement \cite{neitzelsparse}, and many others. Often, the measurement $y_\zeta$ is a noisy perturbation of a ground truth signal $y_0$, that is $y_\zeta = y_0  +\zeta$. The goal is to reconstruct the optimal $\mu$ determining its sparsity properties. 
In the formal limit $\left(\|\zeta\|_Y, \lambda\right) \rightarrow(0,0)$, solutions to  \ref{eq:softproblem} converge to solutions to the
\emph{hard-constrained} interpolation problem:
\begin{align}\label{eq:hard_prob}
    \inf _{\mu \in M(X): K \mu=y_0} \|\mu\|_{\mathrm{TV}}. \tag{$\mathcal{P}_0(y_0)$}
\end{align}
Existence of minimizers is guaranteed for all the problems (see for instance \cite[Section $3$]{DelGrande}). Moreover, this convergence can be made rigorous as in the following classical result \cite[Theorem 3.5]{Hofmann}.
\begin{proposition}\label{prop:hofmann}
 Suppose that there exists $\mu \in M(X)$ such that $K\mu = y_0$. Consider sequences $\zeta_n\in Y$ and $\lambda_n>0$ such that  
 \begin{align}
     \|\zeta_n\|_Y \to 0,\qquad \lambda_n\to 0,\qquad \frac{\|\zeta_n\|_Y^2}{\lambda_n}\to 0 \quad \text{as } n\to\infty.
 \end{align}
Then, any sequence of minimizers of 
\begin{align}
    \inf_{\mu \in M(X)} \frac{1}{2}\|K\mu - {y_{\zeta_n}}\|_Y^2 + \lambda_n\|\mu\|_{\rm TV}  
\end{align}
has a weak* converging subsequence and each limit is a solution to \ref{eq:hard_prob}. In particular, if \ref{eq:hard_prob} has a unique minimizer, then the whole sequence converges weak* to it. 
\end{proposition}

%Note that the sublevel sets of $ \|\mu\|_{\mathrm{TV}}$ are weak* compact by Banach-Alaoglu theorem, and in particular the ball $B=\left\{\mu \in M(X):\|\mu\|_{\mathrm{TV}} \leqslant 1\right\}$ is weak* compact, non-empty, and convex. It is also well known that the extreme points of $B$ are exactly Dirac deltas (see for example \cite[Proposition 4.1]{bc}), that is
%$$
%\operatorname{Ext}(B)=\left\{\sigma \delta_x: x \in X, \sigma \in\{-1,1\}\right\}.
%$$
%Moreover, such set is weak* closed. 

\subsubsection{Duality theory and optimality conditions}

A crucial source of information for minimizers of \ref{eq:softproblem} and \ref{eq:hard_prob} is given by respective dual problems and optimality conditions. 
The dual problem associated with \ref{eq:softproblem} (cf. \cite{Peyre,Temam}) is
\begin{equation}\label{eq:dualsoftproblem}
    \max _{p\in Y: \left\|K_* p\right\|_{\infty} \leqslant 1}\langle y_\zeta, p\rangle-\frac{\lambda}{2}\|p\|_Y^2 \tag{$\mathcal{D}_{\lambda}(y_\zeta)$}, 
\end{equation}
while the dual problem associated with \ref{eq:hard_prob} is
\begin{equation}\label{eq:dualhardproblem}
    \sup _{p\in Y: \left\|K_* p\right\|_{\infty} \leqslant 1}\left\langle y_0, p\right\rangle. \tag{$\mathcal{D}_0(y_0)$}
\end{equation}
Here $K_*:Y\to C(X)$ denotes the adjoint operator defined by
$\langle K\mu,p\rangle_Y=\int_X (K_*p)(x)\,d\mu(x)$ for all $\mu\in M(X)$ and $p\in Y$,
and $\|\cdot\|_\infty$ is the supremum norm on $C(X)$.

\emph{Strong duality} holds between \ref{eq:softproblem} and \refeq{eq:dualsoftproblem} and the existence of $\mu_{\lambda,\zeta} \in M(X)$ solution to \ref{eq:softproblem} and $p_{\lambda,\zeta} \in Y$ solution to \ref{eq:dualsoftproblem}, is equivalent to the following optimality conditions: 
\begin{equation}\label{oc_soft}
    \left\{\begin{aligned}
K_* p_{\lambda,\zeta} & \in \partial \|\mu_{\lambda,\zeta}\|_{\mathrm{TV}}, \\
-p_{\lambda,\zeta} & =\frac{1}{\lambda}\left(K \mu_{\lambda,\zeta}-y_{\zeta}\right) .
\end{aligned}\right.
\end{equation}
Similarly, strong duality holds between \ref{eq:hard_prob} and \refeq{eq:dualhardproblem} and the existence of $\mu_0 \in M(X)$ solution to \ref{eq:hard_prob} and $p_0 \in Y$ solution to \ref{eq:dualhardproblem}, is equivalent to the following optimality conditions:
\begin{equation}\label{oc_hard}
    \left\{\begin{aligned}
K_* p_0 & \in \partial \|\mu_0\|_{\mathrm{TV}}, \\
K\mu_0 & =y_0.
\end{aligned}\right.
\end{equation}
If $\eta_{\lambda,\zeta}=K_* p_{\lambda,\zeta}$ and $\eta_{\lambda,\zeta} \in \partial \|\mu_{\lambda,\zeta}\|_{\mathrm{TV}}$, we call $\eta_{\lambda,\zeta}$ a \emph{dual certificate} for $\mu_{\lambda,\zeta}$ with respect to \ref{eq:softproblem}, since the optimality conditions \eqref{oc_soft} guarantee that $\mu_{\lambda,\zeta}$ is a solution to \ref{eq:softproblem}. Similarly, if $\eta_0=K_* p_0$ and $\eta_0 \in \partial \|\mu_{0}\|_{\mathrm{TV}}$, we call $\eta_0$ a \emph{dual certificate} for $\mu_0$ with respect to \ref{eq:hard_prob}, since the optimality conditions \eqref{oc_hard} guarantee that $\mu_{0}$ is a solution to \ref{eq:hard_prob}.

The solution $p_{\lambda,\zeta}$ to \ref{eq:dualsoftproblem} is unique, since this problem can be reformulated as the Hilbert projection of $y_{\zeta} / \lambda$ onto the closed convex set $\left\{p \in Y: \|K_* p\|_{\infty}\leqslant 1\right\}$, c.f. \cite[Section 2.3]{Peyre}. On the other hand, the solution to \ref{eq:dualhardproblem} is not unique. Among the dual certificates, an important role is played by the minimal-norm dual certificate defined as follows.

\begin{definition}[Minimal-norm dual certificate]\label{def:min_norm_dc}
    Suppose that there exists a solution to \ref{eq:dualhardproblem}. Then the minimal-norm dual certificate  is defined as $\eta_0=K_* p_0$, where 
\begin{equation}
p_0=\operatorname{argmin}\left\{\|p\|_Y: p \in Y \text{ is a solution to } \mathcal{D}_0(y_0)\right\}.
\end{equation}
\end{definition}
Since the subdifferential of the total variation norm can be characterized as 
\begin{align}
   \partial \|\mu\|_{\mathrm{TV}} = \left\{\varphi \in C(X): \|\varphi\|_\infty \leq 1 \text{ and } \int_X \varphi(x)\, d\mu = \|\mu\|_{\mathrm{TV}}\right\}, 
\end{align}

one can deduce the following well-known result that relates the extreme values of the dual certificate with the support of the optimal measure. 

\begin{Lemma}\label{lemma:support}
    Let $\eta_0 \in C(X)$  be a dual certificate for $\mu_0$. Then, $\mu_0$ is a solution to \ref{eq:hard_prob} and 
    \begin{align}
        \supp  \subset \{x \in X: |\eta_0(x)| = 1\}.
    \end{align}
    Similarly, if $\eta_{\lambda,\zeta} \in C(X)$  is the dual certificate for $\mu_{\lambda,\zeta}$. Then,  $\mu_{\lambda,\zeta}$ is a solution to \ref{eq:softproblem} and 
    \begin{align}
        {\rm Supp}(\mu_{\lambda,\zeta}) \subset \{x \in X: |\eta_{\lambda,\zeta}(x)| = 1\}.
    \end{align}
\end{Lemma}

\section{Dual regions and sparsity properties for infinite-width neural networks}

The goal of this section is to construct and analyze dual certificates for infinite-width neural networks with ReLU activation, and to infer sparsity properties of the minimizers. From now on, we will identify $Y=\R^n$ and consider data points $x^j\in\mathbb R^{d+1}$, $j=1,\ldots,n$. We set $y_0:=(y_0^1,\ldots,y_0^n)$ and $y_\zeta:=(y_\zeta^1,\ldots,y_\zeta^n)$. We will consider both the interpolation problem without noise
\begin{align}\label{eq:interpolationproblem}
    \inf_{\mu \in M(\mathbb{S}^{d-1})} \|\mu\|_{\mathrm{TV}} \quad \text{subjected to} \quad F_\mu(x^j) = y_0^j \quad \forall j
\tag{$\mathscr{P}_0(y_0)$}
\end{align}
and the TV-regularized empirical loss where the labels are perturbed as $y_\zeta^j = y_0^j + \zeta^j$, with $\zeta=(\zeta^1,\ldots,\zeta^n)\in\mathbb{R}^n$,
\begin{align}\label{eq:training}
    \inf_{\mu \in M(\mathbb{S}^{d-1})} \frac{1}{2}\sum_{j=1}^n |F_\mu(x^j) - y_\zeta^j|^2 + \lambda  \|\mu\|_{\mathrm{TV}}.
\tag{$\mathscr{P}_{\lambda}(y_{\zeta})$}
\end{align}
We recall that $\|\cdot\|_{\mathrm{TV}}$ denotes the total variation norm, defined in \eqref{eq:tv}. 
%We consider the space $ M(\mathbb{S}^{d-1})$ of Radon measures 
%and $Y := \RR^n$, where $\mathbb{S}^{d-1}$ is the unit $d-$sphere and we note that
%that can be parametrized in polar coordinates by
%\begin{equation}\label{ps}
  % \begin{array}{r}
%w_1=\cos \theta, \\
%w_2=\sin \theta, \\
%0 \leqslant \theta \leqslant 2 \pi.
%\end{array}
%\end{equation} 
%\begin{itemize}
%\item $ M(\mathbb{S}^{d-1})$ endowed with the TV-norm
 %\begin{align*}
   %  \|u\|_{\mathcal{M}(\mathbb T)}= \|\mu\|_{M(\mathbb T)}+ \|\nu\|_{M(\mathbb T)}
% \end{align*}
% is a Banach space whose pre-dual is $C(\mathbb{S}^{d-1})$, the space of continuous functions, that is $M(\mathbb{S}^{d-1}) \simeq C^*(\mathbb{S}^{d-1})$;
%  \item The regularizer $R:M(\mathbb{S}^{d-1})\rightarrow [0,+\infty]$  is the total variation norm defined as
%$$
%\|\mu\|_{\mathrm{TV}}=\sup \left\{\int_{\mathbb{S}^{d-1}} \phi(x) \mathrm{d} \mu(x): \phi \in C(\mathbb{S}^{d-1}),\|\phi\|_{\infty} \leqslant 1\right\},
%$$
%which is a convex, weak* lower semi-continuous and positively 1-homogeneous functional. We also recall that, since the total variation is a sublinear function, its subdifferential becomes
%\begin{equation}\label{eq:subTV}
%\partial\|\mu\|_{\mathrm{TV}}=\left\{\psi \in C(\mathbb{S}^{d-1}) :\|\psi\|_{\infty} \leqslant 1 \text { and } \int_{\mathbb{S}^{d-1}} \psi \mathrm{d} \mu=\|\mu\|_{\mathrm{TV}}\right\}.
%\end{equation}
%\end{itemize} 
Note that by defining the linear operator $K: M(\mathbb{S}^{d-1}) \rightarrow \RR^n$ as
\begin{align}
(K \mu)_j:= F_{\mu}(x^j)= \int_{\mathbb{S}^{d-1}} \sigma(\langle w, x^j\rangle) \mathrm{d} \mu(w) \quad j=1,\ldots, n,
\end{align}
where $\sigma(z) = {\rm ReLU}(z)$, we obtain that \ref{eq:interpolationproblem} has the form of a general hard-constrained TV-regularized problem \ref{eq:hard_prob} 
\begin{align}
   \inf_{\mu \in M(\mathbb{S}^{d-1}): K\mu = y_0} \|\mu\|_{\mathrm{TV}} 
\end{align} 
and \refeq{eq:training}  has the form of a general Tikhonov-type regularized problem \ref{eq:softproblem}
\begin{align}
   \inf_{\mu \in M(\mathbb{S}^{d-1})} \frac{1}{2}\|K\mu - y_\zeta\|_{\R^n}^2 + \lambda  \|\mu\|_{\mathrm{TV}}, 
\end{align} 
 allowing us to use the dual certificate theory developed in Section \ref{sec:tvreg}, provided $K$ is weak*-to-strong continuous. This is verified in the next easy lemma.

\begin{Lemma}\label{lem:predual}
    The operator $K: M(\mathbb{S}^{d-1}) \rightarrow \mathbb{R}^n$ is weak*-to-strong continuous. Moreover, its adjoint $K_* : \mathbb{R}^n \rightarrow C(\mathbb{S}^{d-1})$ is given by 
    \begin{align}\label{eq:predual}
        K_*p(w) = \sum_{j=1}^n p^j \sigma(\langle w, x^j\rangle)\quad \forall p\in \RR^n.
    \end{align}
\end{Lemma}

\begin{proof}
    The weak* continuity of $K$ is immediate. In particular, since the codomain is finite-dimensional, $K$ is weak*-to-strong continuous.
 To prove the characterization of $K_*$ it is enough to note that, for all $\mu\in M(\mathbb{S}^{d-1})$, it holds that \begin{equation}\label{inverseK}
\begin{aligned}
     \int_{\mathbb{S}^{d-1}}K_*p\dm \mu(w)&=\langle K_*p,\mu\rangle=(p,K\mu)_{\RR^n}=\sum_{j=1}^n p^j(K\mu)_j\\
     &=\sum_{j=1}^n p^j \int_{\mathbb{S}^{d-1}} \sigma(\langle w, x^j\rangle) \mathrm{d} \mu(w)=\int_{\mathbb{S}^{d-1}} \sum_{j=1}^n p^j \sigma(\langle w, x^j\rangle) \mathrm{d} \mu(w)\\
     &=\langle\sum_{j=1}^n p^j \sigma(\langle w, x^j\rangle),\mu\rangle.
\end{aligned}
\end{equation}
\end{proof}
We now specialize, for the sake of clarity in the rest of the paper, the abstract dual problems from the previous section to the present ReLU setting.
In particular, using \eqref{eq:predual}, the dual of the interpolation problem \ref{eq:interpolationproblem} becomes
\begin{align}\label{eq:dual-interp}
\sup_{p\in\R^n} \langle p,y_0\rangle
\quad\text{subject to}\quad
\sup_{w\in\mathbb{S}^{d-1}}
\Bigl|\sum_{j=1}^n p^j \,\sigma(\langle w,x^j\rangle)\Bigr| \leq 1.
\tag{$\mathscr{D}_0(y_0)$}
\end{align}
Any maximizer $p_0$ of \refeq{eq:dual-interp} yields a dual certificate $\eta_0(w) := K_*p_0(w)
= \sum_{j=1}^n p_0^j \,\sigma(\langle w,x^j\rangle).$
Instead, the dual of \ref{eq:training} becomes
\begin{align}\label{eq:dual-training}
\sup_{p\in\R^n} \langle y_\zeta,p\rangle - \frac{\lambda}{2} \|p\|^2
\quad\text{subject to}\quad
\sup_{w\in\mathbb{S}^{d-1}}
\Bigl|\sum_{j=1}^n p^j \,\sigma(\langle w,x^j\rangle)\Bigr| \leq 1.
\tag{$\mathscr{D}_\lambda(y_\zeta)$}
\end{align}
Given the maximizer $p_{\lambda,\zeta}$ of \refeq{eq:dual-training}, we define the dual certificate
$\eta_{\lambda,\zeta}(w)
:= K_*p_{\lambda,\zeta}(w)
= \sum_{j=1}^n p_{\lambda,\zeta}^j \,\sigma(\langle w,x^j\rangle)$. Since $Y=\mathbb{R}^n$, the dual problem \ref{eq:dual-training} admits a (unique) maximizer for every $\lambda>0$. 
For \ref{eq:dual-interp}, maximizers may be not unique; we will consider the minimal-norm one when needed.
%Since in this case $Y$ is finite dimensional, a solution of the dual problems \ref{eq:dual-interp} and \ref{eq:dual-training} exist by standard compactness arguments.
Moreover, by the following standard result (see for example \cite[Proposition 1]{Peyre}), the dual certificates converge uniformly in the noiseless regime as the regularization parameter vanishes. 
\begin{proposition}[Convergence of the dual certificates]\label{prop:convdual}
     Let $p_{\lambda,0}$ be the unique solution to $\mathscr{D}_\lambda(y_0)$. Let $p_0$ be the solution with minimal-norm of \ref{eq:dual-interp}. Then,
\begin{align}
\lim _{\lambda \rightarrow 0^{+}}\|\eta_{\lambda,0}-\eta_0\|_{\infty}=0.
\end{align}
\end{proposition}
\begin{proof}
   By definition and \eqref{eq:predual}, for every $w\in\mathbb{S}^{d-1}$ we have
   \begin{align}
       \eta_{\lambda,0}(w)-\eta_0(w)
= \sum_{j=1}^n \bigl(p_{\lambda,0}^j-p_0^j\bigr)\,\sigma(\langle w,x^j\rangle).
   \end{align}
   Hence,
\begin{align*}
\|\eta_{\lambda,0}-\eta_0\|_\infty
&= \sup_{w\in\mathbb{S}^{d-1}}\left|\sum_{j=1}^n \bigl(p_{\lambda,0}^j-p_0^j\bigr)\,\sigma(\langle w,x^j\rangle)\right| \\
&\leq \sum_{j=1}^n |p_{\lambda,0}^j-p_0^j| \, \sup_{w\in\mathbb{S}^{d-1}} \sigma(\langle w,x^j\rangle).
\end{align*}
Since $\sigma(z)=\max\{0,z\}$ and $\sup_{\|w\|=1}\langle w,x^j\rangle=\|x^j\|$, we obtain
$\sup_{w\in\mathbb{S}^{d-1}} \sigma(\langle w,x^j\rangle)=\|x^j\|.$
If we let $C:=\max_{1\leq j\leq n}\|x^j\|<\infty$, then
\begin{equation}
\|\eta_{\lambda,0}-\eta_0\|_\infty
\leq C \sum_{j=1}^n |p_{\lambda,0}^j-p_0^j|
= C\,\|p_{\lambda,0}-p_0\|_1.
\end{equation}
By \cite[Proposition 1]{Peyre}, $p_{\lambda,0}\to p_0$ in $\mathbb R^n$, implying $\|p_{\lambda,0}-p_0\|_1\to 0$ as $\lambda \rightarrow 0$. Therefore, it holds
\begin{align}
\lim_{\lambda\to 0^+}\|\eta_{\lambda,0}-\eta_0\|_\infty = 0
\end{align}
as desired.
    \end{proof}

\subsection{Dual regions}

 Given vectors $\{x^1,\ldots,x^n\}\subset\R^d$ (assume $x^i\neq 0$), we consider the hyperplane arrangement
\[
H_i:=\{w\in\R^d:\langle w,x^i\rangle=0\}, \qquad i=1,\ldots,n,
\]
together with the associated open half-spaces
\[
A_i^{1}:=\{w\in\R^d:\langle w,x^i\rangle>0\}, \qquad
A_i^{0}:=\{w\in\R^d:\langle w,x^i\rangle<0\}.
\]
A binary vector $\pi=(\pi_1,\ldots,\pi_n)\in\{0,1\}^n$ prescribes, for each $i$, a sign pattern for the scalars
$\langle w,x^i\rangle$ and thus determines the intersection
\begin{align}
R_\pi:=\bigcap_{i=1}^n A_i^{\pi_i}.
\end{align} 
We call the (possibly empty) open subset $R_\pi\subset \R^d$ a \emph{dual region}. The term “dual” emphasizes that these regions live in the \emph{weight space} ($w$-variable), in contrast with the
more common “linear region” decompositions in \emph{input space} ($x$-variable). 
Intersections of this type are well known in the literature of neural networks analysis. We refer, for example, to the “sectors” of \cite{Sectors_article,dedios2020sparsity}, the linear regions studied in \cite{montufar2014number}, and the activation regions introduced in \cite{karhadkar2024mildly}. For instance, see Figure~\ref{fig:dualregions_d2}, which depicts the spherical dual regions $R_\pi\cap\mathbb S^{1}$ for $d=n=2$.
\begin{definition}[Dual region]\label{def:dualregion}
Given $\pi=(\pi_1,\ldots,\pi_n)\in\{0,1\}^n$, we call a \emph{dual region} associated with $\pi$ 
\begin{align}
R_{\pi}:=\bigcap_{i=1}^n A_i^{\pi_i}.
\end{align}
\end{definition}

\begin{figure}[t]
\centering
\begin{tikzpicture}[scale=1.1, line cap=round, line join=round]
  \def\R{2.0}     
  \def\a{30}       
  \def\b{70}     

  \coordinate (O) at (0,0);

  \pgfmathsetmacro{\aMinus}{\a-90}
  \pgfmathsetmacro{\aPlus}{\a+90}
  \pgfmathsetmacro{\bMinus}{\b-90}
  \pgfmathsetmacro{\bPlus}{\b+90}

  \begin{scope}
    \clip (O) circle (\R);

    \path[fill=blue!20]   (O) -- (\bMinus:\R) arc (\bMinus:\aPlus:\R) -- cycle;
    \path[fill=teal!18]   (O) -- (\aMinus:\R) arc (\aMinus:\bMinus:\R) -- cycle; 
    \path[fill=red!18]    (O) -- (\aPlus:\R)  arc (\aPlus:\bPlus:\R)  -- cycle; 
    \path[fill=yellow!20] (O) -- (\bPlus:\R)  arc (\bPlus:\aMinus+360:\R) -- cycle; 
  \end{scope}

  % --- circle
  \draw[thick] (O) circle (\R);

  % --- axes
  \draw[->] (-2.5,0) -- (2.5,0);
  \draw[->] (0,-2.5) -- (0,2.5);

  \draw[dashed, thick] (\aPlus:\R) -- (\aMinus:\R)
  node[pos=0.82, below left, inner sep=1pt] {$H_1$};

\draw[dashed, thick] (\bPlus:\R) -- (\bMinus:\R)
  node[pos=0.82, above right, inner sep=1pt, xshift=-1pt, yshift=-1pt] {$H_2$};

 \draw[->, thick] (O) -- (\a:\R) node[above right] {$x^1$};
\draw[->, thick] (O) -- (\b:\R) node[above right] {$x^2$};

  \node at ({(\bMinus+\aPlus)/2}:1.25) {$R_{(1,1)}$};
  \node at ({(\aMinus+\bMinus)/2}:1.35) {$R_{(1,0)}$};
  \node at ({(\aPlus+\bPlus)/2}:1.35) {$R_{(0,1)}$};
  \node at ({(\bPlus+\aMinus+360)/2}:1.35) {$R_{(0,0)}$};

\end{tikzpicture}
\caption{Dual regions on $\mathbb S^1$ induced by two inputs ($d=n=2$).}
\label{fig:dualregions_d2}
\end{figure}
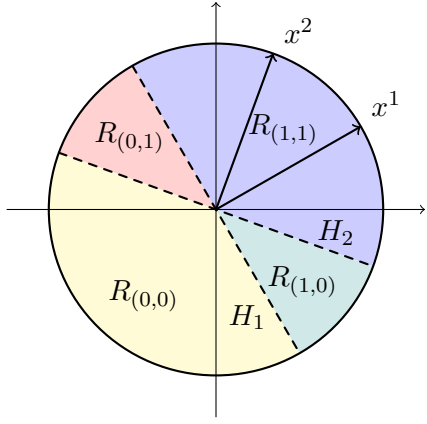
\begin{comment}
\begin{figure}
    \centering
   \begin{tikzpicture}
        \draw (0,0) circle(2cm);
        \fill[teal!40](-1.72,1) arc(150:300:2)--(0,0);
        \fill[blue!40](-1,1.72) arc(120:150:2)--(0,0);
        \fill[red!40](1.72,-1) arc(-30:120:2)--(0,0);
        \fill[yellow!40](1,-1.72) arc(300:330:2)--(0,0);
        \draw[->] (-2.5,0) -- (2.5,0);
        \draw[->] (0,-2.5) -- (0,2.5);
        \draw[->] (0,0)--(1.68, 1) node[below]{$x^1$};
        \draw[->] (0,0)--(1, 1.68) node[left]{$x^2$};
        \draw[dashed](-1,1.72)--(1,-1.72);
        \draw[dashed](-1.72,1)--(1.72,-1);
        \draw (-1,-1) node{$R_\pi$}; 
    \end{tikzpicture}
    \caption{Regions with $d=n=2$.}
    \label{fig:enter-label}
\end{figure}
\end{comment}
It is well known from classical results on linearly separable dichotomies/central hyperplane arrangements (Vapnik–Chervonenkis dimension theory) that the number of distinct binary vectors $\pi$ such that the dual region $R_\pi$ is non-empty is at most
$2\sum_{k=0}^{d-1}\binom{n-1}{k}$; equality holds when the hyperplanes $\{H_i\}_{i=1}^n$ are in general position (see, e.g., \cite[Theorem~1]{Cover1965}). In particular, thanks to the binomial theorem, we obtain that there are at most $2^n$ possible combinations (independent of $d$) if and only if $n\leqslant d$. Instead, if $n>d$, then, one can bound the number of these combinations by $\mathcal{O}(n^d)$ (refer to \cite[Proposition~6]{karhadkar2024mildly}, \cite[Chapter~4]{Sectors_article}). 
In what follows, we will denote by $\Theta$ the set of binary vectors $\pi$ such that the region $R_\pi$ is non-empty. 

Our analysis will necessarily deal with the boundary of dual regions. To capture optimality on these sets, we will introduce the concept of \emph{higher codimension} dual regions, obtained by allowing some of the defining inequalities to become equalities. This yields a natural stratification of the weight space into open cells of varying codimension, which will be crucial for our study. We also remark that the need for such structures has already been noted in \cite{karhadkar2024mildly,Sectors_article}.
%Since $w\in \mathbb{S}^{d-1}$ we can use the parametrization \eqref{ps} to obtain, in a fixed region $R_{\pi}$, the following identity:
%\begin{equation}\label{eta0}
   % \bar{\eta}_0(\theta)=\sum_{i=1}^n p_0^i\pi_i(\cos\theta x_{i,1}+\sin \theta x_{i,2})=c_{0,k}^{1}\cos\theta+c_{0,k}^{2}\sin\theta,%=\sum_{i=1}^n p_0^i\pi_ix^i^2+w_1\sum_{i=1}^n p_0^i\pi_i(x^i^1-x^i^2).
%\end{equation}
%where $c_{0,k}^{1}=\sum_{i=1}^n p_0^i\pi_ix_{i,1}\neq 0$ and $c_{0,k}^{2}=\sum_{i=1}^n p_0^i\pi_ix_{i,2}\neq 0$.
%We can observe that $\bar{\eta}_0$ is differentiable in the region $R_{\pi}$ (maybe we can say $\eta_0$ is piecewise differentiable as in \cite{Sectors_article}), and the critical points are:
%\begin{equation}
  %  \eta'_0=-c_{0,k}^{1}\sin \theta+c_{0,k}^{2}\cos\theta=0\Leftrightarrow \theta=\tan^{-1}\left(\frac{c_{0,k}^{1}}{c_{0,k}^{2}}\right).
%\end{equation}
%Therefore the uniqueness of the maximum or the minimum in every region is already given. 
%\textcolor{red}{
%\begin{Lemma}
%Two regions $R_{\pi_1}$ and $R_{\pi_1}$ of codimension $n-k$ have boundary in common if and only if 
%\begin{align}
%    \{i: \pi_1(i) = 0\}  = \{i: \pi_2(i) = 0\}
%\end{align}
%and there exists only one $\hat i \notin \{i: \pi_1(i) = 0\}$ such that $\pi_1(\hat i) \neq \pi_2(\hat i)$.
%\end{Lemma}
%}
Consider the index set $I:=\{1,\ldots,n\}$ and, for $k\leq n$, let $J^k=\{j_1,\ldots,j_k\}\subset I$. We aim to define regions of codimension $n-k$ that belong to the topological boundary of a given region $R_\pi$.
\begin{definition}[Higher–codimension dual regions]\label{def:hig_cod_reg}
   Let $\pi\in\{0,1\}^n$ and $R_\pi$ be as in Definition~\ref{def:dualregion}. 
   %Given $k\leq n$ and $J^k \subset I$, we define the $J^k$ boundary of $R_\pi$ as follows:
%\begin{equation}\label{boundaryregions}
%R^{J^k}_{\pi}:=\bigcap_{j\in J^k}A_{j}^{\pi_j} \bigcap_{z\in I\setminus J^k}\left\{w \in \mathbb{R}^d:\langle w, x^z\rangle= 0\right\}
  % \end{equation}
%with the convention that $R_\pi^{I} = R_\pi$.
For $J^k\subset I$ with $|J^k|=k$,  we define the $J^k$-stratum of $R_\pi$ as follows:
\begin{equation}\label{boundaryregions}
R^{J^k}_\pi :=  \Big(\bigcap_{j\in J^k} A_j^{\pi_j}\Big)\;\cap\; \Big(\bigcap_{z\in I\setminus J^k} H_z\Big),
   \end{equation}
where $H_z:=\{w\in\R^d:\langle w,x^z\rangle=0\}$. By convention, $R_\pi^{I}=R_\pi$.
\end{definition}
The set $R^{J^k}_\pi$ lies in $L^{J^k}:=\bigcap_{z\in I\setminus J^k} H_z$, which has codimension at most $|I\setminus J^k|=n-k$ in $\mathbb R^d$
(and equals $n-k$ if the normals $\{x^z\}_{z\in I\setminus J^k}$ are linearly independent).
Similarly to the codimension zero case, we denote by 
\begin{align}
\Theta^{J^k}:=\{\pi\in\{0,1\}^n:\;R^{J^k}_\pi\neq\emptyset\}
\end{align}
the set of sign patterns for which the stratum $R^{J^k}_{\pi}$ is non-empty. For simplicity, we are going to call $R^{J^k}_{\pi}$ again a dual region.
\begin{remark}[Boundary stratification]\label{boundary}
    Let $\overline{R}^{J^{k}}_{\pi}$ denote the closure of $R^{J^{k}}_{\pi}$ in the relative topology of
$L^{J^k}$. Here and below, $\partial$ denotes the boundary taken in the same relative topology. Note that the topological boundary of a region $R_\pi$ can be simply rewritten as \begin{equation}\label{full_boundary}
     \partial R_\pi = \bigcup_{J^{n-1}\subset I} \overline R^{J^{n-1}}_\pi,
    \end{equation}
    that is, the union of the closures of its codimension $1$ faces. Lower dimensional strata are contained in these closures.
    More generally, for $k=1,\ldots,n-2$, the topological boundary of $R_\pi^{J^k}$ becomes:
\begin{align}\label{boundary_k}
        \partial R^{J^k}_\pi = \bigcup_{J^{k-1}\subset J^k} \overline R^{J^{k-1}}_\pi,
    \end{align}
    where $|J^{k-1}|=k-1$.
\end{remark}

\subsection{Sparse characterization and upper bounds on the number of Dirac deltas}

The derivation of an exact reconstruction result for the interpolation hard-constrained problem \refeq{eq:interpolationproblem} is based on the behaviour of the dual certificate in the dual regions, determined by the input data $\{x^1,\ldots,x^n\}$. For this, it is crucial to choose the ReLU as activation function $\sigma$. First, it is easy to see that since  $\sigma(z) = {\rm ReLU}(z)$, for every $i =1,\ldots,n$, we can write
\begin{equation}\label{eq:ee}
    \sigma(\langle w, x^i\rangle)=\pi_i\langle w, x^i\rangle \quad \text{for every }  w \in R_\pi.
\end{equation}

As a consequence the minimal-norm dual certificate can be written as a linear function in each dual region $R_\pi$ as stated in the next lemma. When $\langle w,x^i\rangle=0$ we have $\sigma(\langle w,x^i\rangle)=0$; boundary cases will be treated via the strata $R_\pi^{J^k}$.

\begin{Lemma}\label{lem:dualinregion}
    Given a binary vector $\pi$, for every $w\in R_{\pi}$, every dual certificate $\eta_0$ can be rewritten as
    \begin{equation}
        \eta_0(w)= \sum_{i=1}^n p_0^i\pi_i \langle w, x^i \rangle. 
    \end{equation}
\end{Lemma}
\begin{proof}
The lemma follows directly from Lemma \ref{lem:predual} and \eqref{eq:ee}.
%One computation shows the result:
%\begin{equation}
%\begin{aligned}
%    \eta_0(w)&=p_0^1\pi_1\langle w, x^1\rangle+\cdots+p_0^n\pi_n\langle w, x^n\rangle\\
%    &=\langle w,\sum_{i=1}^n p_0^i\pi_ix^i\rangle \\
%    &=w_1\sum_{i=1}^n p_0^i\pi_ix^i_{1}+\ldots+w_d\sum_{i=1}^n p_0^i\pi_ix^i_{d}\\
%    &=\sum_{i=1}^n p_0^i\pi_i(w_1x^i_{1}+\dots+w_dx^i_{d}),
%    \end{aligned}
%\end{equation}
%where $\pi_1,\ldots, \pi_n\in \{0,1\}$. 
\end{proof}

In particular, every dual certificate $\eta_0\in C(\mathbb S^{d-1})$ is the restriction to $\mathbb S^{d-1}$ of a function that is linear on each cone $R_\pi$ and thus globally piecewise linear. This enables us to use the dual certificate theory summarized in Section \ref{sec:tvreg} to obtain information about the support of the minimizers of \ref{eq:interpolationproblem} and \ref{eq:training}. 
Indeed, in the next lemma, we use the piecewise linear structure of the dual certificate to show that the number of points where its absolute value achieves one is finite. Thanks to the optimality conditions in \eqref{oc_hard}, this will imply that the support of the minimizers of \ref{eq:interpolationproblem} is a discrete set (see Corollary \ref{cor:exactreconstruction}).

\begin{Lemma}\label{unique_c.p.}
Let $p_0$ be a solution to \ref{eq:dual-interp} and let $\eta_0=K_*p_0$ be a dual certificate.
Fix $\pi\in\{0,1\}^n$, $J^k\subset I$ with $|J^k|=k$, and set
\[ 
z_0:=\sum_{i\in J^k} p_0^i\pi_i x^i,
\qquad 
\tilde z_0:=P_{L^{J^k}} z_0,
\]
where $P_{L^{J^k}}$ denotes the orthogonal projection onto $L^{J^k}$.
If there exists $w_0\in R_{\pi}^{J^k}\cap \mathbb S^{d-1}$ such that $|\eta_0(w_0)|=1$, then $\tilde z_0\neq 0$ and
\[
w_0 = \mathrm{sign}(\eta_0(w_0))\,\frac{\tilde z_0}{\|\tilde z_0\|}.
\]
Moreover, $|\eta_0(w)|<1$ for every $w\in \big(\overline{R}_{\pi}^{J^k}\cap\mathbb S^{d-1}\big)\setminus\{w_0\}$.
\end{Lemma}
\begin{proof}
First, note that on $R_{\pi}^{J^k}\cap\mathbb S^{d-1}$ the dual certificate $\eta_0$ can be written as
\[
\eta_0(w)=\sum_{i\in J^k} p_0^i\pi_i\langle w,x^i\rangle
=\langle z_0,w\rangle,
\qquad z_0:=\sum_{i\in J^k} p_0^i\pi_i x^i.
\]
Since $w\in L^{J^k}$, we also have $\langle z_0,w\rangle=\langle P_{L^{J^k}}z_0,w\rangle=\langle \tilde z_0,w\rangle$,
where $\tilde z_0:=P_{L^{J^k}}z_0$.
Assume there exists $w_0\in R_{\pi}^{J^k}\cap\mathbb S^{d-1}$ such that $|\eta_0(w_0)|=1$.
Then $\tilde z_0\neq 0$ (otherwise $\eta_0\equiv 0$ on $L^{J^k}\cap\mathbb S^{d-1}$).
By Cauchy--Schwarz, for every $w\in L^{J^k}\cap\mathbb S^{d-1}$,
\[
|\eta_0(w)|=|\langle \tilde z_0,w\rangle|\leq \|\tilde z_0\|,
\]
with equality if and only if $w=\pm \tilde z_0/\|\tilde z_0\|$.
Since $|\eta_0(w_0)|=1$, necessarily $\|\tilde z_0\|\geq  1$, and equality in Cauchy--Schwarz implies
\[
w_0=\mathrm{sign}(\eta_0(w_0))\,\frac{\tilde z_0}{\|\tilde z_0\|}.
\]
To prove uniqueness inside $\overline{R}_{\pi}^{J^k}\cap\mathbb S^{d-1}$, suppose that
$|\eta_0(\bar w)|=1$ for some $\bar w\in \overline{R}_{\pi}^{J^k}\cap\mathbb S^{d-1}$.
Then by the same argument $\bar w=\pm \tilde z_0/\|\tilde z_0\|$.
The choice $\bar w=-w_0$ is incompatible with belonging to $\overline{R}_{\pi}^{J^k}$.
Indeed, for any index $i\in J^k$ with $\pi_i=1$ we have $\langle w_0,x^i\rangle\geq  0$ on the closure,
hence $\langle -w_0,x^i\rangle\leq 0$, which would force $-w_0$ to lie in the opposite half-space (or on the hyperplane),
contradicting $\bar w\in \overline{R}_{\pi}^{J^k}$ unless all such scalar products vanish (which cannot occur if $|\eta_0(w_0)|=1$).
Therefore $\bar w=w_0$, and the statement follows.
\end{proof}

\begin{remark}
Note that the previous lemma put constraints on the points where the absolute value of the dual certificate reaches one. In particular, under the assumption of the previous proposition it also holds that 
$|\eta_0(w)|<1$ for every $w\in \bigcup_{J^{k-1}\subset J^k} \overline R_\pi^{J^{k-1}} \cap \mathbb{S}^{d-1}$.
Moreover, if $w_0$ belongs to a region of $0$-codimension, that is if $k = n$, then $\tilde z_0 = z_0$.  
\end{remark}

A similar statement holds for the dual certificate associated with \ref{eq:training}.

\begin{Lemma}\label{unique_c.p.soft}
    Let $p_{\lambda,\zeta} \in \R^n$ be the solution to \ref{eq:dual-training} and let $\eta_{\lambda,\zeta} = K_* p_{\lambda,\zeta}$ be the dual certificate for \ref{eq:training}. 
    Fix $\pi\in\{0,1\}^n$, $J^k\subset I$ with $|J^k|=k$, and set
    \begin{align}\label{eq:gated_vector_1}
    \ds{z_{\lambda,\zeta}=\sum_{i\in J^k} p_{\lambda,\zeta}^i\pi_ix^{i}}, \quad \tilde z_{\lambda,\zeta} = P_{L^{J^k}} z_{\lambda,\zeta}.
    \end{align}
    
    If there exists $w_{\lambda,\zeta}\in  R_{\pi}^{J^k} \cap \mathbb{S}^{d-1}$  such that $|\eta_{\lambda,\zeta}(w_{\lambda,\zeta})| = 1$, then $\tilde z_{\lambda,\zeta} \neq 0$ and
    \begin{align}\label{eq:gated_vector_2}
        w_{\lambda,\zeta} = \mathrm{sign}(\eta_{\lambda,\zeta}(w_{\lambda,\zeta})) \frac{\tilde z_{\lambda,\zeta}}{\|\tilde z_{\lambda,\zeta}\|}.
    \end{align}
    Moreover, $|\eta_{\lambda,\zeta}(w)|<1$ for every $w \in (\overline R^{J^k}_{\pi} \cap \mathbb{S}^{d-1}) \setminus \{w_{\lambda,\zeta}\}$.
\end{Lemma}

From the previous lemmas it follows directly the mentioned sparse characterization of minimizers of \ref{eq:interpolationproblem} and \ref{eq:training}.

\begin{theorem}[Sparsity of minimizers]\label{cor:exactreconstruction}
The following two statements hold:
\begin{itemize}
    \item  There exists $N \in \N$ such that all the solutions to \ref{eq:interpolationproblem} are of the form 
    \begin{align}
        \mu_0 = \sum_{i=1}^N c^i_0 \delta_{w^i_0},
    \end{align}
    where $c^i_0 \in \R$ and $w^i_0  = \mathrm{sign}(\eta_0(w_0)) \frac{z_0}{|z_0|}$ with $z_0 =\sum_{i \in J^k} p_0^i \pi_i x^i$ for some binary vector $\pi$ and $J^k \subset I$.
    \item There exists $M \in \N$ such that all the solutions to \ref{eq:training} are of the form 
    \begin{align}
        \mu_{\lambda,\zeta} = \sum_{i=1}^M c_{\lambda,\zeta}^i \delta_{w_{\lambda,\zeta}^i},
    \end{align}
    where $c_{\lambda,\zeta}^i \in \R$ and $w_{\lambda,\zeta}^i  = \mathrm{sign}(\eta_{\lambda,\zeta}(w^i_{\lambda,\zeta})) \frac{z_{\lambda,\zeta}}{\|z_{\lambda,\zeta}\|}$ with $z_{\lambda,\zeta} =\sum_{i \in J^k} p_{\lambda,\zeta}^i \pi_i x^i$ for some binary vector $\pi$ and $J^k \subset I$.
\end{itemize}
   
\end{theorem}
\begin{proof}
    The proof follows from Lemma \ref{unique_c.p.} and Lemma \ref{unique_c.p.soft} together with the optimality conditions for \ref{eq:interpolationproblem} and \ref{eq:training} summarized in Lemma \ref{lemma:support}.
\end{proof}

From a rough estimate one can upper bound both $N$ and $M$ (we denote by $N_{max}$ a common upper bound depending only on $n$ and $d$) in the previous theorem by
\begin{align}\label{eq:roughbound}
    N_{max} \leq
3^n.
\end{align}
This is obtained as all possible ways to choose the sign of each scalar product $\langle w,x^i\rangle$ for $i=1,\ldots,n$ ($\langle w,x^i\rangle>0$, $\langle w,x^i\rangle<0$ or $\langle w,x^i\rangle=0$ in correspondence with regions of higher codimension). We are counting the ternary activation patterns for each $i\in\{1,\ldots,n\}$.

However, finer estimates are possible, as described in the remaining part of this subsection.

\begin{Lemma}\label{lem:crudebound}
    The absolute value of the dual certificate  $\eta_0$ (resp. $\eta_{\lambda,\zeta}$) for \ref{eq:interpolationproblem} (resp. \ref{eq:training}) achieves its maximum in at most 
    \begin{align}
        \max_{0\leq k \leq n}\binom{n}{k}\Xi(n-k,d-k)
    \end{align}
    points, where $\Xi(n-k,d-k)$ denotes the number of non-empty regions induced in $\R^{d-k}$ by an arrangement of $n-k$ hyperplanes and we use the convention that   $\Xi(n-k,d-k) = 0$ when $d-k < 0$.
\end{Lemma}

\begin{proof}
  The proof of the statement is the same for the interpolation problem \ref{eq:interpolationproblem} and for the training problem \ref{eq:training}. The proof is based on the observation that if $|\eta|$ is maximized at some point in a given stratum $R^{J^k}_\pi$ for $J^k \subset I$, then it cannot be maximized at any other point in lower dimensional strata $\overline R_\pi^{J^{k-1}}$ for $J^{k-1} \subset J^k$. Moreover, distinct strata correspond to disjoint ternary patterns, so maximizers are counted by enumerating all the strata. In particular, we have $\binom{n}{k}$ ways of choosing $k$'s $x^i$ such that $\langle x^i,w\rangle = 0$. This intersection of planes determines a set of dimension $\R^{d-k}$. In such a set, we are allowed to choose the intersection of $n-k$ hyperplanes that leads to the factor $\Xi(n-k,d-k)$.  
\end{proof}

\begin{remark}\label{rem:betterbound}
The previous lemma gives a sharper estimate with respect to the rough bound in \eqref{eq:roughbound}. Indeed, since $\Xi(n-k,d-k) \leq 2^{n-k}$, then by the binomial theorem we obtain that 
\begin{align}
        \max_{0\leq k \leq n}\binom{n}{k}\Xi(n-k,d-k) \leq \max_{0\leq k \leq n}\binom{n}{k} 2^{n-k}  \leq \sum_{k=0}^n \binom{n}{k} 2^{n-k} = 3^n 
\end{align}
Therefore, if $n\leqslant d$ the best bound we can get is 
\begin{align}
    \max_{0\leq k \leq n}\binom{n}{k} 2^{n-k}
\end{align}
that can be estimated grossly by $3^n$. Even if slightly more precise estimates can be obtained, we will not enter into these details here. 
Instead, if $n>d$ it holds that 
\begin{align}
    \Xi(n-k,d-k) \leq O((n-k)^{d-k}).
\end{align}
Hence, taking into account that  $\Xi(n-k,d-k) = 0$ when $d-k < 0$, it holds that
\begin{align}
 \max_{0\leq k \leq n}& \binom{n}{k}\Xi(n-k,d-k)  \leq  \max_{0\leq k \leq n}\binom{n}{k} O((n-k)^{d-k}) =\max_{0\leq k \leq n}\binom{n}{n-k} O(k^{d-(n-k)})  \\
 & \leq \max_{0\leq k \leq n}\binom{n}{n-k} O(n^{d-(n-k)}) \leq \frac{n^{n-k}}{(n-k)!} O(n^{d-(n-k)}) \leq O(n^{d}).
\end{align}
Note that this bound, in the case $n>d$, is actually equal to $O(\Xi(n,d))$, that is, proportional to the number of dual regions. This yields a sharper estimate than the one provided in \cite{dedios2020sparsity}.

\end{remark}

    \subsection{Linear independence and uniqueness of the minimizers} \label{uniq}
In the previous subsection we proved that both problems \ref{eq:interpolationproblem} and \ref{eq:training} admit sparse minimizers and we provided an upper bound on the number of Dirac deltas in any such representation. We now address uniqueness. As common in the available literature on TV-regularized optimization problems in the space of measures, uniqueness follows from a combination of (i) a localization property of the dual certificate and (ii) a linear independence condition on the measurement vectors $\{K\delta_{w^i}\}$. We provide sufficient conditions ensuring (ii) in our ReLU setting. We caution that our conditions are likely not sharp and that a complete characterization lies beyond the scope of this work.

 %However, the uniqueness of such sparse minimizers has not been discussed yet. As common in the available literature on TV-regularized optimization problems in the space of measures, uniqueness is consequence of linear independence properties on the measurements $K$. In this subsection we provide conditions that aim at ensuring that. 

We begin by recalling several definitions from linear algebra.
\begin{definition}[Hadamard product]\label{Hadamard}
    The Hadamard or entry-wise product of two matrices $M$ and $N$ of the same dimension is defined as $M \odot N:=\left(M_{i j} \cdot N_{i j}\right)$. 
\end{definition}
We will also need the definition of \emph{transversal} given for example in \cite{brualdi1991combinatorial}.
\begin{definition}[Transversal]\label{Transversal}
   A transversal of an $m\times n$ matrix is a set of $m$ entries, one from each row and no two from the same column. 
\end{definition}

We also recall that the determinant of 
a square matrix $M=\left(m_{i j}\right) \in \mathbb{R}^{n \times n}$ can be computed as 
\begin{align}\label{eq:determinant}
\operatorname{det}(M):=\sum_{ p\in \mathcal{P}_n} \left(\operatorname{sign}(p) \prod_{i=1}^n m_{i, p(i)}\right),
\end{align}

where the sum is over all permutations $p\in \mathcal{P}_n$ of $n$ elements.
\begin{definition}[Permanent]\label{permanent}
    Given a square matrix $M=\left(m_{i j}\right) \in \mathbb{R}^{n \times n}$, the permanent of $M$ is defined as
$$
\operatorname{perm}(M):=\sum_{ p\in \mathcal{P}_n} \prod_{i=1}^n m_{i, p(i)}.
$$
\end{definition}

Let $\mu_0=\sum_{i=1}^N c_0^i \delta_{w_0^i}$ with $c_0^i \in \R$ and $w_0^i \in R_{\pi^i}$ for $i=1,\ldots,N$. Let $w_0:=(w_0^1,\ldots,w_0^N)$. We define the $N\times n$ evaluation matrix 
    \begin{equation}\label{Kmat}
    \mathbf K(w_0):=\left[\begin{array}{cccc}
\sigma(\langle w_0^1, x^1\rangle)  & \cdots & \sigma(\langle w_0^1, x^n\rangle)\\
\vdots  & \cdots & \vdots \\
\vdots  & \ddots & \vdots \\
\sigma(\langle w_0^N, x^1\rangle) & \cdots & \sigma(\langle w_0^N, x^n\rangle)
\end{array}\right]= \left[\begin{array}{cccc}
\pi_1^1\langle w_0^1, x^1\rangle & \cdots & \pi_n^1\langle w_0^1, x^n\rangle\\
\vdots  & \cdots & \vdots \\
\vdots& \ddots & \vdots \\
\pi^N_1\langle w_0^N, x^1\rangle & \cdots & \pi^N_n \langle w_0^N, x^n\rangle
\end{array}\right],
     \end{equation}
     where $\pi_j^i := \mathbf 1_{\langle w_0^i,x^j\rangle>0}$ so that 
$\sigma(\langle w_0^i,x^j\rangle)=\pi_j^i\langle w_0^i,x^j\rangle$. Its $i$-th row is exactly the measurement vector $(K\delta_{w_0^i})^\top \in \mathbb R^{1\times n}$.
Moreover, if we denote
    \begin{equation}
    \Pi:= \left[\begin{array}{cccc}
\pi_1^1\ & \cdots & \pi_n^1 \\
\vdots  & \cdots & \vdots \\
\vdots& \ddots & \vdots \\
\pi^N_1& \cdots & \pi^N_n
\end{array}\right],
    X:=\left[\begin{array}{cccc}
x^1_1  & \cdots & x^n_1\\
\vdots  & \cdots & \vdots \\
\vdots  & \ddots & \vdots \\
x^1_d & \cdots & x^n_d
\end{array}\right],  W:=\left[\begin{array}{cccc}
w^1_{0,1}  & \cdots & w^N_{0,1}\\
\vdots  & \cdots & \vdots \\
\vdots& \ddots & \vdots \\
w^1_{0,d} & \cdots & w^N_{0,d}
\end{array}\right],
     \end{equation}
 where $w_0^i = (w_{0,1}^i, \ldots, w_{0,d}^i)$ and  $x^i = (x_1^i, \ldots, x_d^i)$ then we can rewrite $
    \mathbf K(w_0)= \Pi\odot W^{\top} X.$
   In the following lemma we will prove that $\{K \delta_{w_0^i}\}_{i=1}^N$ are linearly independent under certain assumptions on $\Pi$ and supposing that $N \leq n$. Note that such condition is necessary, since otherwise the vectors $\{K\delta_{w_0^i}\}_{i=1}^N$ cannot be linearly independent.

    \begin{Lemma}\label{lem:lin_ind_meas}
        Let $\mu_0=\sum_{i=1}^N c_0^i \delta_{w_0^i}$ with $c_0^i \in \R$ and $w_0^i \in R_{\pi^i}$, $N \leq n$. Suppose that the following conditions are satisfied:
        \begin{itemize}
            \item[i)] Some $N\times N$ minor of $\Pi$, called $\Pi_N$, has full rank;
            \item[ii)] $\operatorname{perm}(\Pi_N)=|\det(\Pi_N)|$.

       % \item[ii)] $\operatorname{perm}(\Pi_N)=|\operatorname{det}(\Pi_N)|$. 
            %\item $\left\{w_0^i\right\}_{i=1}^N$ are linearly independent and $\operatorname{Span}\left(\left\{w_0^i\right\}_{i=1}^N\right) \cap \operatorname{Span}\left(\left\{x^j\right\}_{j=1}^n\right)^{\perp}=\{0\}$;
           % \item  %for every $i=1,\ldots,N$ there exists at least one $j=1,\ldots,n$ such that $(w_0^i,x^j)>0$.
          %  there exists at least one $i=1,\ldots,N$ such that $\pi^i_j=1$ for every $j=1,\ldots,n$,
        \end{itemize}
        Then, the vectors $\{K \delta_{w_0^i}\}_{i=1}^N$ are linearly independent.
    \end{Lemma}
  
    \begin{proof}
   By the definition of $K$ we have to prove that $\{\sigma(\langle w_0^i, x^1\rangle),\dots, \sigma(\langle w_0^i, x^n\rangle)\}_{i=1}^N$ are linearly independent. 
   %Since 
   %  \begin{align}
   %      \{\sigma(\langle w_0^i, x^1\rangle),\dots, \sigma(\langle w_0^i, x^n\rangle)\}_{i=1}^N = \{\pi^i_1 \langle w_0^i, x^1\rangle),\dots, \pi^i_n \sigma(\langle w_0^i, x^n\rangle)\}_{i=1}^N
   %  \end{align}
%that is
%\begin{equation}
%    \left\{\begin{array}{c}
%\sum_{i=1}^N \alpha_i \sigma(\langle w_0^i, x^1\rangle)=0 \\
%\vdots \\
%\sum_{i=1}^N \alpha_i \sigma(\langle w_0^i, x^n\rangle)=0
%\end{array} \Leftrightarrow \alpha_1=\cdots=\alpha_N=0.\right.
%\end{equation}
%We are excluding the region with all coefficients $\pi^i_j=0$ for every $j=1,\ldots,n$. We can rewrite the previous system of equations in the following way:
%\begin{equation}
%    \left\{\begin{array}{c}
%\sum_{i=1}^N \alpha_i \pi^i_1 \langle w_0^i, x^1\rangle=0 \\
%\vdots \\
%\sum_{i=1}^N \alpha_i \pi^i_n \langle w_0^i, x^n\rangle=0,
%\end{array} \right.
%\end{equation}
%where each $\pi^i$ is associated with a different region (here region means closed region). 
We will equivalently prove that $\mathbf K(w_0)$ has rank $N$, that is it has an $N\times N$ minor with nonzero determinant.
Let $\Pi_N$ be an $N\times N$ minor of $\Pi$ which is nonsingular by assumption \emph{i)}, and let $(W^\top X)^N$ be the corresponding $N\times N$ minor of $W^\top X$ (same rows and columns). Set
\begin{align}
B := (W^\top X)^N,\qquad A := \Pi_N\in\{0,1\}^{N\times N}.
\end{align}
Then the corresponding minor of $\mathbf K(w_0)=\Pi\odot (W^\top X)$ is $A\odot B$ and
\begin{equation}
\det(A\odot B)=\sum_{p\in\mathcal P_N}\operatorname{sign}(p)\prod_{i=1}^N A_{i,p(i)}\,B_{i,p(i)}.
\end{equation}
By assumption \emph{i)}, there exists at least one permutation $\bar p$ such that $\prod_i A_{i,\bar p(i)}=1$ (a transversal of ones). Moreover, by assumption \emph{ii)}, or equivalently $\operatorname{perm}(A)=|\det(A)|$, all permutations $p$ with $\prod_i A_{i,p(i)}=1$ have the same parity, hence the same sign $\operatorname{sign}(p)=s\in\{\pm1\}$. For those $p$, we also have $B_{i,p(i)}=\langle w_0^i,x^{p(i)}\rangle>0$, since $A_{i,p(i)}=1$ means $\langle w_0^i,x^{p(i)}\rangle>0$ in the associated dual region. Therefore, every nonzero term in the expansion of $\det(A\odot B)$ has the same sign $s$ and is strictly positive in magnitude, and in particular the term corresponding to $\bar p$ is nonzero due to the fact that $\prod_i A_{i,\bar p(i)}=1$. Hence, $\det(A\odot B)\neq 0$, implying that $\mathbf K(w_0)$ has rank $N$.
\end{proof}

\begin{remark}
Note that in the proof we are using that the second assumption, that is $\operatorname{perm}(\Pi_N)=|\det(\Pi_N)|$ for $0$-$1$ matrices, is equivalent to all transversals of $\Pi_N$, i.e. all permutations $p$ with $\prod_{i=1}^N (\Pi_N)_{i,p(i)}=1$, having the same parity. 
The previous lemma provides easy scenarios where the linear independence of  $\{K \delta_{w_0^i}\}_{i=1}^N$ is ensured. For example if $\Pi_N$ is a permutation matrix, then $i)$ and $ii)$ are automatically verified. To the best of our knowledge, these conditions have not been explored in the literature. Extending this line of results is left for future research.
\end{remark}

We are now ready to examine uniqueness for the sparse solutions to \ref{eq:interpolationproblem} and \ref{eq:training}. We first treat the interpolation problem  \ref{eq:interpolationproblem}, since the analysis for the training problem \ref{eq:training} follows the same strategy. 

Let $\mu_0=\sum_{i=1}^N \sigma^i c_0^i\,\delta_{w_0^i}$ with $c_0^i>0$, $\sigma^i\in\{-1,+1\}$ and $w_0^i\in \mathbb{S}^{d-1}$ satisfy the hard constraint $K\mu_0=y_0$ of \ref{eq:interpolationproblem}. 
 If $\mu_0$ is a minimizer, the dual certificate $\eta_0$ satisfies $\eta_0(w_0^i)=\sigma^i$ by Lemma \ref{lemma:support}. To conclude uniqueness, we require that $|\eta_0|$ attains its maximum $1$ only at the atoms $\{w_0^i\}_{i=1}^N$ and that the measurements $\{K\delta_{w_0^i}\}_{i=1}^N$ are linearly independent (ensured by $(i),(ii)$ in Lemma \ref{lem:lin_ind_meas}).

\begin{theorem}\label{thm:uniqueness_hard_sol}
     Let $\mu_0=\sum_{i=1}^N c_0^i \sigma^i \delta_{w_0^i}$ be such that $y_0=K \mu_0$, with $c_0^i>0$, $\sigma^i \in \{-1,1\}$ and $w_0^i \in \overline{R}_{\pi^i}$ for all $i=1, \ldots, N$. Let $\eta_0$ be the dual certificate for $\mu_0$, and suppose that $\{K \delta_{w_0^i}\}_{i=1}^N$ are linearly independent and 
     \begin{itemize}
    \item[$(\mathrm{LC})$] $|\eta_0(w)| < 1$ for all $w \in \mathbb{S}^{d-1} \setminus \{w_0^1, \ldots w_0^N\}$.
\end{itemize}
Then, $\mu_0$ is the unique solution to \ref{eq:interpolationproblem}.
\end{theorem}
\begin{proof}
    The proof follows directly by applying  \cite[Lemma 4.5]{DelGrande}.
\end{proof}
A verbatim statement holds for the training problem when the dual certificate for \ref{eq:training} plays the role of $\eta_0$.

\begin{theorem}\label{thm:uniqueness_soft}
     Let $\mu_{\lambda,\zeta}=\sum_{i=1}^N c_{\lambda,\zeta}^i \sigma^i \delta_{w^i_{\lambda,\zeta}}$, with $c_{\lambda,\zeta}^i >0$, $\sigma^i \in \{-1,1\}$ and $w^i_{\lambda,\zeta} \in \overline{R}_{\pi^i}$ for all $i=1, \ldots, N$. Let $\eta_{\lambda,\zeta}$ be the dual certificate for $\mu_{\lambda,\zeta}$, and suppose that $\{K \delta_{w^i_{\lambda,\zeta}}\}_{i=1}^N$ are linearly independent and 
     \begin{itemize}
    \item[$(\widetilde{\mathrm{LC}})$] $|\eta_{\lambda,\zeta}(w)| < 1$ for all $w \in \mathbb{S}^{d-1} \setminus \{w^1_{\lambda,\zeta}, \ldots w^N_{\lambda,\zeta}\}$.
\end{itemize}
Then, $\mu_{\lambda,\zeta}$ is the unique solution to \ref{eq:training}.
\end{theorem}

\begin{remark}
  The localization condition (LC), together with the optimality conditions ensures that the dual certificate is maximized only at the Dirac deltas locations $\{w_0^i\}_{i=1}^N$ (and similarly for $\eta_{\lambda,\zeta}$). By Lemmas~\ref{unique_c.p.}--\ref{unique_c.p.soft}, for each dual region $R_{\pi}^{J^k}\cap\mathbb S^{d-1}$ the function $|\eta|$ has at most one point where it can possibly attain its maximum (when the region is nonempty). In particular, in the interior of a region there cannot be two distinct maximizers. Consequently, to check the localization condition it is not necessary to control $|\eta|$ everywhere on $\mathbb S^{d-1}$ but it suffices to inspect the finite collection of \emph{candidate maximizers} and verify that $|\eta|=1$ only on the support of the sparse measure, while all the remaining candidates satisfy $|\eta|<1$. 
The size of this candidate set is finite and is bounded as in Lemma~\ref{lem:crudebound} and Remark~\ref{rem:betterbound}.
\end{remark}

\section{Sparse stability for noisy labels under regularization}\label{sec:stability}

%In this section, we aim at showing that for small regularization parameter and in a low noise regime, we expect all the solutions to \ref{eq:training} to have the same number of Dirac deltas as the solution to \ref{eq:interpolationproblem}. Moreover, their locations and weights are continuous perturbation of the locations and the weights of the solution to   \ref{eq:interpolationproblem}. 
%These type of results are called \emph{exact support recovery} \cite{Peyre}.
In this section, we aim to show that for sufficiently small regularization parameter $\lambda$
and in a low noise regime, every solution to \ref{eq:training} has the same number of Dirac
deltas as the (unique) solution to \ref{eq:interpolationproblem}. Moreover, the locations and
weights are continuous perturbations of those of $\mu_0$ as $(\lambda,\zeta)\to(0,0)$ within the
admissible set $N_{\alpha,\lambda_0}$ defined below. Such results are typically referred to as
\emph{exact support recovery} \cite{Peyre}.

We consider the following set of admissible parameters/noise levels for $\lambda_0>0$ and $\alpha>0$:
\begin{equation}\label{eq:noise_set}
N_{\alpha,\lambda_0}
=\left\{(\lambda,\zeta)\in \mathbb{R}_{+}\times\mathbb{R}^n:
0<\lambda\leq \lambda_0 \ \text{ and }\ \|\zeta\|\leq \alpha\lambda\right\}.
\end{equation}
This set was introduced in \cite{Peyre} and describes the regime in which one expects exact
support recovery.
Let us also define the extended support as introduced in \cite{Peyre}.
\begin{definition}[Extended support \cite{Peyre}]
Let $\mu_0 \in M(\mathbb{S}^{d-1})$ be such that $y_0=K \mu_0$ and let $\eta_0 \in C(\mathbb{S}^{d-1})$ denote the minimal-norm dual certificate associated with \ref{eq:interpolationproblem}. The extended support of the measure $\mu_0$ is defined as 
\begin{equation}
\operatorname{Ext}(\mu_0):=\left\{w \in \mathbb{S}^{d-1} : |\eta_0(w)|= 1\right\}.
\end{equation}
\end{definition}
Note that $\operatorname{Supp}(\mu_0)\subset \operatorname{Ext}(\mu_0)$ by Lemma~\ref{lemma:support}. Analogously, for $\mu_{\lambda,\zeta}$ with dual certificate $\eta_{\lambda,\zeta}$, set $\operatorname{Ext}( \mu_{\lambda,\zeta}):=\{w \in \mathbb{S}^{d-1} : |\eta_{\lambda,\zeta}(w)|= 1\}$. 
We now recall a fundamental result, \cite[Lemma $1$]{Peyre}, regarding the support of solutions to \ref{eq:training} in the low noise regime and for small regularization parameters.
 \begin{Lemma}\label{le:local_support}
     Let $p_0$ be the minimal norm solution to \ref{eq:dual-interp}, $p_{\lambda,\zeta}$ the solution to \ref{eq:dual-training} and $\mu_{\lambda,\zeta} \in M(\mathbb{S}^{d-1})$ any solution to \ref{eq:training}. Given $\varepsilon, \delta>0$, there exist $\alpha>0, \lambda_0>0$ such that, for all $(\lambda, \zeta) \in N_{\alpha, \lambda_0}$, we have \begin{equation}\label{eq:supp_noise_meas}
        \|p_{\lambda,\zeta}-p_0\|\leqslant \delta \quad \text{ and } \quad  \operatorname{supp} \mu_{\lambda,\zeta} \subset \operatorname{Ext}^{\varepsilon}\left(\mu_0\right),
     \end{equation}
     where $\operatorname{Ext}^{\varepsilon}\left(\mu_0\right):=\bigcup_{w \in \operatorname{Ext}\left(\mu_0\right)} B_{\varepsilon}(w)$.
 \end{Lemma}

 \begin{proof}
     The proof can be adapted straightforwardly from  \cite[Lemma $1$]{Peyre} by taking into account Proposition \ref{prop:convdual}.  
 \end{proof}

We now prove that the solutions to \ref{eq:training} are sparse, for small enough parameters/noise levels in the set $N_{\alpha,\lambda_0}$. In particular, the number of atoms equals the cardinality of the support of the measure $\mu_0$ and their locations are uniquely determined.
Note that this statement does not imply uniqueness of the corresponding coefficients and consequently of the representation. This problem will be addressed in the next section.
%We first prove a geometric lemma that provide information on the neighboring regions.
%
%\begin{Lemma}[Neighboring regions of co-dimension zero]
%    Let $R_{\pi_1}$, $R_{\pi_2} \in \Theta$ two distinct regions  of co-dimension zero. Suppose that there exists exactly one $j \in I$ such that 
%    \begin{align}
%        \pi_1(j) \neq \pi_2(j).
%    \end{align}
%    Then $R^{J^{n-1}}_{\pi_1} \cap R^{J^{n-1}}_{\pi_2} \neq \emptyset$ where
%\begin{align}
%    J^{n-1} = I \setminus \{j\} 
%\end{align}
%Viceversa if $R^{J^{n-1}}_{\pi_1} \cap R^{J^{n-1}}_{\pi_2} \neq \emptyset$, then 
%    \begin{align}
%        \pi_1(j) \neq \pi_2(j)
%    \end{align}
%only in $j = I \setminus  J^{n-1}$.
%\end{Lemma}
%\begin{proof}
%    Suppose that there exists exactly one $j \in I$ such that $\pi_1(j) \neq \pi_2(j)$. Then pick $w$ such that $\langle w,x^j \rangle = 0$ and 
%    $\langle w,x^i \rangle > 0$  for all $i\neq j, \pi_1(i) = 1$ and $\langle w,x^i \rangle < 0$ for all $i\neq j, \pi_1(i) = 0$. Note that such $w$ exists since $R_{\pi_1}, R_{\pi_2} \in \Theta$ and thus they are non-empty. By the way in which  regions are defined it holds that $w \in \partial R_{\pi_1} \cap \partial R_{\pi_2}$. 
%\end{proof}
%
%
%\begin{Lemma}[Neighboring regions of any co-dimension]
% Let $R^{J_1}_{\pi_1}$, $R^{J_2}_{\pi_2} \in \Theta$ two distinct regions. Suppose that there exists exactly one $j \in I$ such that 
%    \begin{align}
%        \pi_1(j) \neq \pi_2(j).
%    \end{align}
%    Then $\partial R_{\pi_1} \cap \partial R_{\pi_2} \neq \emptyset$.
%\end{Lemma}

Before stating the next proposition, we introduce an auxiliary set that describes the common boundary between two different regions.
\begin{definition}
Let $\pi\in\{0,1\}^n$ and let $J^k\subset I$.
Define the compatible family of strata as
\begin{align}\label{eq:higher_reg}
\mathcal{F}_{\pi}^{J^k}
:=\left\{ R^{J}_{\bar \pi} : R^{J}_{\bar\pi}\neq\emptyset,\ J^k\subset J,\ \text{and } \bar\pi_j=\pi_j\ \forall j\in J^k\right\}.
\end{align}
\end{definition}
%We then define the following transition index set:
%\begin{align}
%    \mathcal{H}_{\pi}^{J^k}=\{ j\in I: \pi^1_j\neq \pi^2_j \text{ with }  R^{J_1}_{\pi^1}, R^{J_2}_{ \pi^2} \in \mathcal{F}_{\pi}^{J^k}\}.
%\end{align}
%
%In particular, we will use $\mathcal{H}$ to denote the index set over which the data points $x^j$ are linearly independent, that is, we assume that the linear independence of the vectors $x^j$ holds only when restricted to the indices in $\mathcal{H}$.
%

As typical in super-resolution problems, we will need to impose a further non-degeneracy condition for the dual certificate that is related with its strong convexity (resp. concavity) in the minimum (resp. maximum) points.

\begin{assumption}[Boundary non-degeneracy condition] We define the following non-degeneracy condition on points that belongs to the boundary of the dual regions:
\begin{itemize}
\item[(ND)]
 Let $\mu_0=\sum_{i=1}^N c_0^i\sigma^i\delta_{w_0^i}$, with $c_0^i>0$, $\sigma^i\in\{-1,+1\}$,
and assume $w_0^i\in R^{J^k}_{\pi}\cap\mathbb S^{d-1}$ for some $k<n$. Given $p_0$ the minimal norm solution of \ref{eq:dual-interp},
for any $R^{J}_{\bar\pi}\in\mathcal F^{J^k}_\pi$, define
\begin{equation}
z_0(J,\bar\pi):=\sum_{j\in J}p_0^j\,\bar\pi_j\,x^j,
\qquad
\tilde z_0(J,\bar\pi):=P_{J}\,z_0(J,\bar\pi).
\end{equation}
We say that $\mu_0$ satisfies (ND) if for any two distinct strata $R^{J_1}_{\pi^1},R^{J_2}_{\pi^2}\in\mathcal F^{J^k}_\pi$ such that
$\tilde z_0(J_\ell,\pi^\ell)\neq 0$, it holds that
\begin{equation}
\frac{\tilde z_0(J_1,\pi^1)}{\|\tilde z_0(J_1,\pi^1)\|}
\neq
\frac{\tilde z_0(J_2,\pi^2)}{\|\tilde z_0(J_2,\pi^2)\|}.
\end{equation}
\end{itemize}
\end{assumption}

We are now ready to state and prove the main result of this section.

\begin{proposition}\label{unique_loc}
Let $N \leq n$. Assume that $\mu_0=\sum_{i=1}^N c_0^i \sigma^i \delta_{w_0^i} $, where $c_0^i>0$, $\sigma^i \in\{-1,+1\}$ and $w_0^i \in  \overline R^{J_i}_{\pi^i}$ for all $i=1, \ldots, N$, satisfies \emph{(LC)} and \emph{(ND)}. Let $\{K\delta_{w_0^i}\}_{i=1}^N$  be linearly independent and $\mu_{\lambda,\zeta} \in M(\mathbb{S}^{d-1})$ be any solution to \ref{eq:training}. Then, there
exist $\alpha>0, \lambda_0>0$ such that, for all $(\lambda, \zeta) \in N_{\alpha, \lambda_0}$, there exists a unique collection of $w_{\lambda,\zeta}^i\in \mathbb{S}^{d-1}$ such that
\begin{align}\label{eq:rep}
\mu_{\lambda,\zeta}=\sum_{i=1}^N c_{\lambda,\zeta}^i \sigma^i\delta_{w_{\lambda,\zeta}^i},
\end{align}
where $c_{\lambda,\zeta}^i>0$  and $\eta_{\lambda,\zeta}(w_{\lambda,\zeta}^i)=\sigma^i$ for every $i=1, \ldots, N$. Moreover, as $\lambda\to 0$ with $\|\zeta\|\leq \alpha\lambda$, it holds
\begin{equation}
w^i_{\lambda,\zeta}\to w_0^i,
\qquad
c^i_{\lambda,\zeta}\to c_0^i.
\end{equation}

\end{proposition}
\begin{proof}
Without loss of generality, assume that $\sigma^i=1$ for all $i$. The argument is local around each $w_0^i$, and the case $\sigma^i=-1$ is handled in the same way by replacing $\eta$ with $-\eta$ on the corresponding subset of indices.

By (LC), $|\eta_0(w)|<1$ for every $w\neq w_0^i$, hence for sufficiently small $\varepsilon_0>0$,
\begin{align}
\Ext(\mu_0)\cap B_{\varepsilon_0}(w_0^i)=\{w_0^i\}\qquad\text{for every }i=1,\dots,N. 
\end{align}
Moreover, by Lemma \ref{unique_c.p.},  if $w_0^i\in R^{J_i}_{\pi^i}$ then 
\begin{equation}\label{eq:w0-representation}
w_0^i
=
\frac{\tilde z_0^i}{\|\tilde z_0^i\|},
\qquad
\tilde z_0^i
:=
P_{J_i}\Big(\sum_{j\in J_i} p_0^j \pi^i_j x^{j}\Big),
\end{equation}

where $p_0$ is the dual maximizer in \ref{eq:dual-interp}.\\
Fix $i$ and let $R^{J^k}_\pi=R^{J_i^k}_{\pi^i}$ be the $J^k$-stratum such that $w_0^i\in R^{J^k}_\pi$ with $k\leq n$.
We want to prove that for $(\lambda,\zeta) \in N_{\alpha, \lambda_0}$ with $\alpha, \lambda_0$ sufficiently small,
the set $\Ext(\mu_{\lambda,\zeta})\cap B_{\varepsilon_0}(w_0^i)$
contains at most one point.
Indeed, by Lemma~\ref{le:local_support} the support of $\mu_{\lambda,\zeta}$ lies in $\bigcup_{m=1}^N B_{\varepsilon_0}(w_0^m)$ for $(\lambda,\zeta) \in N_{\alpha, \lambda_0}$ with $\alpha, \lambda_0$ sufficiently small.
Thus, showing that $\Ext(\mu_{\lambda,\zeta})\cap B_{\varepsilon_0}(w_0^i)$ has at most one point
implies that $\mu_{\lambda,\zeta}$ has at most one atom in $B_{\varepsilon_0}(w_0^i)$.
%Given $ w_0^i \in R^{J^k_i}_{\pi^i}$ in a certain $k$ substratum, we want to prove that for sufficiently small $\varepsilon$ it holds that
%$|\eta_{\lambda,\zeta}(w)| = 1$ in at most one $ w \in B_{\varepsilon}(w_0^i)$. For notational convenience, we fix an $i$ and write $w_0^i\in R^{J^k}_{\pi}$ with $k\leq n$. 

Note that with this choice of $J^k$ and $\pi$, shrinking $\varepsilon_0$ if necessary, the ball $B_{\varepsilon_0}(w_0^i)$ can intersect only strata in the compatible family $\mathcal F^{J^k}_{\pi}$ as defined in \eqref{eq:higher_reg}.
%Let $p_0$ be the minimal-norm solution to \ref{dualhard_prob_ReLU}. By choosing $\varepsilon$ small enough, we claim that 
%\begin{align}\label{eq:diffzero}
%    p_0^j \neq 0 \quad \forall j \in I \setminus J^k.
%\end{align}
%Indeed, if such $j$ exists, then for every $w \in R^{J^k \cup \{j\}}_{\bar \pi}$ with $\bar \pi(j) = 1$ it holds that 
%\begin{align}
 %   \eta_0(w) = \sum_{i \in J^k \cup \{j\}} p_0^i \pi_i \langle x_i, w\rangle = \sum_{i \in J^k} p_0^i \pi_i \langle x_i, w\rangle = \eta_0(w^i_0)  
%\end{align}
%respectively, this would contradict the LC condition by taking $\varepsilon$ small enough. 
Arguing as in Lemma \ref{lem:dualinregion}, we can write the dual certificate $\eta_{\lambda,\zeta}$ on the dual region $R^{J}_{\bar\pi}\in\mathcal{F}^{J^k}_\pi$ as
\begin{equation}
   {\eta}_{\lambda,\zeta}(w)=\sum_{j \in J}p_{\lambda,\zeta}^j \bar\pi_j \langle w, x^j\rangle,
    \qquad w\in R^{J}_{\bar\pi},
\end{equation}
where $p_{\lambda,\zeta} \in \R^n$ is the solution to the dual problem \ref{eq:dual-training}. Since $R^J_{\bar\pi}\subset L^{J}$, for $w\in R^J_{\bar\pi}\cap\mathbb S^{d-1}$ we have
\begin{equation}
\eta_{\lambda,\zeta}(w)=\langle P_J\Big(\sum_{j\in J} p_{\lambda,\zeta}^j\bar\pi_j x^j\Big),\, w\rangle
=\langle \tilde z_{\lambda,\zeta}(J,\bar\pi),w\rangle.
\end{equation}
Hence, $|\eta_{\lambda,\zeta}(w)|\leq \|\tilde z_{\lambda,\zeta}(J,\bar\pi)\|$ and equality can hold only if
$w=\pm \tilde z_{\lambda,\zeta}(J,\bar\pi)/\|\tilde z_{\lambda,\zeta}(J,\bar\pi)\|$.
At most one of these two points belongs to the fixed sign region $R^J_{\bar\pi}$. Therefore, following the computation made in the proof of Lemma \ref{unique_c.p.} for $\eta_0$, in each region $R^{J}_{\bar\pi}\cap\mathbb S^{d-1}$ with $R^{J}_{\bar\pi}\in\mathcal F^{J^k}_\pi$,
the set $\Ext(\mu_{\lambda,\zeta})$ contains at most one point, and whenever it is nonempty it must be of the form
\begin{equation}
w_{\lambda,\zeta}(J,\bar\pi)=\frac{\tilde z_{\lambda,\zeta}(J,\bar\pi)}{\|\tilde z_{\lambda,\zeta}(J,\bar\pi)\|}.
\end{equation}
 %Therefore, repeating the computation made in the proof of Lemma \ref{unique_c.p.} for $\eta_0$, we obtain that in each region $ R^{J}_{\bar \pi} \cap \mathbb{S}^{d-1}$ with $R^{J}_{\bar  \pi} \in \mathcal{F}^{J^k}_\pi$ there is \emph{at most one} point of $\Ext(\mu_{\lambda,\zeta})$, and it must be of the form
%\begin{equation}
%w_{\lambda,\zeta}(J,\bar\pi)
%=
%\frac{\tilde z_{\lambda,\zeta}(J,\bar\pi)}{\|\tilde z_{\lambda,\zeta}(J,\bar\pi)\|},
%\qquad
%\tilde z_{\lambda,\zeta}(J,\bar\pi)
%:=
%P_{J}\Big(\sum_{j\in J} p_{\lambda,\zeta}^j\,\bar\pi_j\,x^j\Big).
%\end{equation}
%It remains to prove that there cannot be more than on element in $\Ext(\tilde \mu_\lambda)$ belonging to $B_\varepsilon(w_0^i)$. This is true inside the region $R_{\pi^i}$ thanks to Proposition \ref{unique_c.p.}, but we still have to prove it when the point $w_{\lambda,\zeta}^i$ belongs to the boundary $\partial R_{\pi^i}$ (which contains every possible subregion of higher codimension).  
Now we want to prove that there exists  $\alpha,\lambda_0$  sufficiently small such that for all $(\lambda,\zeta) \in N_{\alpha,\lambda_0}$ there exists at most one point $w \in B_{\varepsilon_0}(w_0^i)$ with  $|\eta_{\lambda,\zeta}(w)|=1$. Assume by contradiction that this does not hold. Then, also recalling the first step of the proof, there exist a vanishing sequence $(\alpha_n, \lambda_{0,n})$, $(\lambda,\zeta) \in N_{\alpha_n,\lambda_{0,n}}$ and two distinct regions $R^{J_1}_{\pi_1}, R^{J_2}_{\pi_2} \in \mathcal{F}^{J^k}_\pi$ with points $w_1 \in R^{J_1}_{\pi_1}$, $w_2 \in R^{J_2}_{\pi_2}$ such that $w_1, w_2 \in \operatorname{Ext}(\mu_{\lambda,\zeta})\cap B_{\varepsilon_0}(w_0^i)$,  $w_1\neq w_2$ (note that for easing the notation we are omitting the dependence of $n$ in $(\lambda,\zeta)$ and in the selected dual regions). Since $R^{J_1}_{\pi_1}$ and $R^{J_2}_{\pi_2}$ are assumed distinct, we have that $(J_1,\pi^1)\neq (J_2, \pi^2)$.
For $\ell\in\{1,2\}$ set
\begin{equation}
z_{\lambda,\zeta}^\ell:=\sum_{j\in J_\ell} p_{\lambda,\zeta}^j\,\pi^\ell_j\,x^j,
\qquad
\tilde z_{\lambda,\zeta}^\ell:=P_{J_\ell}\, z_{\lambda,\zeta}^\ell.
\end{equation}
Applying the previous step of the proof again and taking $\alpha_n$, $\lambda_{0,n}$ small enough,  $\tilde z_{\lambda,\zeta}^\ell\neq 0$ and
\begin{equation}
w_\ell=\frac{\tilde z_{\lambda,\zeta}^\ell}{\|\tilde z_{\lambda,\zeta}^\ell\|},\qquad \ell=1,2.
\end{equation}
For $\ell\in\{1,2\}$ define the linear map
$
\tilde z^\ell(p):=P_{J_\ell}\Big(\sum_{j\in J_\ell} p^j\,\pi_j^\ell\,x^j\Big)$ and fix an arbitrary $0< \varepsilon < \varepsilon_0$. Since $\tilde z^\ell(p_0)\neq 0$, the map
$p\mapsto \tilde z^\ell(p)/\|\tilde z^\ell(p)\|$ is continuous at $p_0$.
Hence, there exists $\delta>0$ such that
\begin{equation}
\|p-p_0\|\leq \delta
\quad\Longrightarrow\quad
\left\|
\frac{\tilde z^\ell(p)}{\|\tilde z^\ell(p)\|}
-
\frac{\tilde z^\ell(p_0)}{\|\tilde z^\ell(p_0)\|}
\right\|\leq \varepsilon,
\qquad \ell=1,2.
\end{equation}
Applying Lemma~\ref{le:local_support} with this $\delta$, there exist $n$ sufficiently big such that $\|p_{\lambda,\zeta}-p_0\|\leq \delta$.
Therefore, 
\begin{equation}
\left\|
\frac{\tilde z_{\lambda,\zeta}^\ell}{\|\tilde z_{\lambda,\zeta}^\ell\|}
-
\frac{\tilde z_0^\ell}{\|\tilde z_0^\ell\|}
\right\|\leq \varepsilon,
\qquad \ell=1,2,
\end{equation}
where $\tilde z_0^\ell:=P_{J_\ell}\big(\sum_{j\in J_\ell} p_0^j\,\pi_j^\ell\,x^j\big)$.
Moreover, by again applying Lemma~\ref{le:local_support} with $n$ sufficiently big, it holds that $w_1,w_2\in B_\varepsilon(w_0^i)$ and thus $\|w_2-w_1\|\leq2\varepsilon$.

Therefore, we obtain
\begin{align*}
\left\|
\frac{\tilde z_0^1}{\|\tilde z_0^1\|}
-
\frac{\tilde z_0^2}{\|\tilde z_0^2\|}
\right\|
&\leq
\left\|
\frac{\tilde z_0^1}{\|\tilde z_0^1\|}
-
\frac{\tilde z_{\lambda,\zeta}^1}{\|\tilde z_{\lambda,\zeta}^1\|}
\right\|
+
\left\|
\frac{\tilde z_{\lambda,\zeta}^1}{\|\tilde z_{\lambda,\zeta}^1\|}
-
\frac{\tilde z_{\lambda,\zeta}^2}{\|\tilde z_{\lambda,\zeta}^2\|}
\right\|
+
\left\|
\frac{\tilde z_{\lambda,\zeta}^2}{\|\tilde z_{\lambda,\zeta}^2\|}
-
\frac{\tilde z_0^2}{\|\tilde z_0^2\|}
\right\| \\
&\leq\varepsilon + \|w_1-w_2\| + \varepsilon
\leq\varepsilon + 2\varepsilon + \varepsilon
=4\varepsilon.
\end{align*}
Since (ND) states that the normalized vectors
$\tilde z_0(J,\bar\pi)/\|\tilde z_0(J,\bar\pi)\|$ are pairwise distinct for distinct strata in $\mathcal F^{J^k}_\pi$,
there exists
\[
c:=\min_{\substack{R^{J_1}_{\pi^1},R^{J_2}_{\pi^2}\in\mathcal F^{J^k}_\pi\\ (J_1,\pi^1)\neq (J_2,\pi^2)}}
\left\|
\frac{\tilde z_0(J_1,\pi^1)}{\|\tilde z_0(J_1,\pi^1)\|}
-
\frac{\tilde z_0(J_2,\pi^2)}{\|\tilde z_0(J_2,\pi^2)\|}
\right\|>0.
\]
Choosing $\varepsilon$ such that $\varepsilon <c/4$ yields a contradiction.
 Hence, for each $i$ the following holds: for $(\lambda,\zeta) \in N_{\alpha,\lambda_0}$ with $\lambda_0, \alpha$  small enough there exists at most one point $w\in B_{\varepsilon_0}(w_0^i)$ with $|\eta_{\lambda,\zeta}(w)|=1$.

This also shows that $\mu_{\lambda,\zeta}$ admits the representation \eqref{eq:rep}
with $c_{\lambda,\zeta}^i\geq 0$ and $\eta_{\lambda,\zeta}(w_{\lambda,\zeta}^i)=1$ for all $i$. Indeed, 
By Lemma \ref{le:local_support}, the support of the measure $\mu_{\lambda,\zeta}$ is contained in $\operatorname{Ext}_{\varepsilon_0}\left(\mu_0\right)$. Moreover, optimality implies $\operatorname{supp}\mu_{\lambda,\zeta}\subset \Ext(\mu_{\lambda,\zeta})=\{w:\ |\eta_{\lambda,\zeta}(w)|=1\}$.
Hence, $\mu_{\lambda,\zeta}$ is a finite sum of Dirac masses, with at most one atom in each ball, that is
\begin{equation}\label{eq:sum_diracs_noise}
\mu_{\lambda,\zeta}=\sum_{i=1}^N c_{\lambda,\zeta}^i \delta_{w_{\lambda,\zeta}^i}.
\end{equation}
For each $i$ either $\operatorname{supp}\mu_{\lambda,\zeta}\cap B_{\varepsilon_0}(w_0^i)=\emptyset$,
or there exists a unique point $w_{\lambda,\zeta}^i\in B_{\varepsilon}(w_0^i)$ such that
\begin{equation}
\operatorname{supp}\mu_{\lambda,\zeta} \cap B_{\varepsilon_0}(w_0^i)
=\Ext(\mu_{\lambda,\zeta})\cap B_{\varepsilon_0}(w_0^i)
=\{w_{\lambda,\zeta}^i\}.
\end{equation}
In the latter case, $w_{\lambda,\zeta}^i\in \Ext(\mu_{\lambda,\zeta})$ and $\eta_{\lambda,\zeta}(w_{\lambda,\zeta}^i)=1$. In particular, we may write
\eqref{eq:sum_diracs_noise},
where $c_{\lambda,\zeta}^i\geq 0$ and $c_{\lambda,\zeta}^i>0$ if and only if $\operatorname{supp}\mu_{\lambda,\zeta}\cap B_\varepsilon(w_0^i)\neq\emptyset$.
Thus, it remains to prove that $c_{\lambda,\zeta}^i>0$ for all $i$ when $(\lambda_0,\alpha)$ is sufficiently small.

By the minimality of $\mu_{\lambda,\zeta}$ for \ref{eq:training},
\begin{equation}
\lambda \left\|\mu_{\lambda,\zeta}\right\|_{\rm TV} \leqslant \frac{1}{2}\left\|K \mu_{\lambda,\zeta}-y_0-\zeta\right\|^2+\lambda \left\|\mu_{\lambda,\zeta}\right\|_{\rm TV} \leqslant \frac{\|\zeta\|^2}{2}+\lambda \left\|{\mu}_0\right\|_{\rm TV},
\end{equation}
where in the second inequality, we used that $K \mu_0-y_0=0$. Dividing by $\lambda>0$ and using $\|\zeta\|\leq \alpha\lambda$ for $(\lambda,\zeta)\in N_{\alpha,\lambda_0}$, we obtain a uniform bound
\begin{equation}
\|\mu_{\lambda,\zeta}\|_{\rm TV}
=\sum_{i=1}^N c_{\lambda,\zeta}^i
\leq
\|\mu_0\|_{\rm TV} + \frac{\alpha^2}{2}\lambda_0,
\end{equation}
and since $\|\mu_{\lambda,\zeta}\|_{\rm TV}=\sum_{i=1}^N c_{\lambda,\zeta}^i$ is uniformly bounded, each coefficient $c_{\lambda,\zeta}^i$ is uniformly bounded on $N_{\alpha,\lambda_0}$.
Assume by contradiction that there exist some $j$ and a sequence $(\lambda_k,\zeta_k)\in N_{\alpha_k,\lambda_{0,k}}$ with $\lambda_{0,k}\to 0$ and $\alpha_k \to 0$, such that
\begin{equation}
c_{\lambda_k,\zeta_k}^j \to 0 \qquad\text{as }k\to\infty.
\end{equation}
By the compactness of the sphere and the boundedness of the coefficients, we can extract a subsequence (not relabelled) such that, for each $i$,
\begin{equation}
\delta_{w_{\lambda_k,\zeta_k}^i}\stackrel{*}{\rightharpoonup}\delta_{\hat w_0^i},
\qquad
c_{\lambda_k,\zeta_k}^i\to \hat c_0^i,
\end{equation}
with $\hat c_0^j=0$. In particular, it holds that
\begin{equation}
\mu_{\lambda_k,\zeta_k}
=\sum_{i=1}^N c_{\lambda_k,\zeta_k}^i \delta_{w_{\lambda_k,\zeta_k}^i}
\stackrel{*}{\rightharpoonup}
\sum_{i=1}^N \hat c_0^i \delta_{\hat w_0^i}
=: \hat\mu_0.
\end{equation}

On the other hand, by the stability result \cite[Theorem 3.5]{Hofmann} we know that $\mu_{\lambda_k,\zeta_k}\stackrel{*}{\rightharpoonup}\mu_0$ as $k\to\infty$. By Theorem \ref{thm:uniqueness_hard_sol}, the minimizer $\mu_0$ of \ref{eq:training} is unique, hence necessarily $\hat\mu_0=\mu_0$. Since the balls $\{B_{\varepsilon_0}(w_0^i)\}_{i=1}^N$ are pairwise disjoint and $\operatorname{supp}\mu_{\lambda_k,\zeta_k}\subset\bigcup_i B_{\varepsilon_0}(w_0^i)$ for $k$ large, the labeling of atoms is fixed by the balls.
Since $\mu_0$ is uniquely representable (i.e., both the coefficients $c_0^i$ and the points $w_0^i$ are unique), we deduce, thanks to the complementarity assumption, that 
\begin{equation}
\hat c_0^i = c_0^i>0,
\qquad
\hat w_0^i = w_0^i,
\qquad i=1,\ldots,N,
\end{equation}
which contradicts $\hat c_0^j=0$. Therefore, for $(\lambda_0,\alpha)$ sufficiently small, we have that for all $(\lambda,\zeta) \in N_{\alpha,\lambda_0}$
\begin{equation}
c_{\lambda,\zeta}^i>0 \qquad\text{for all }i=1,\ldots,N
\end{equation}
obtaining the representation \eqref{eq:rep}.

Finally, let $(\lambda_k,\zeta_k) \in N_{\alpha,\lambda_0}$ with $(\lambda_k,\zeta_k)\to(0,0)$. By \cite[Theorem~3.5]{Hofmann} we have
$\mu_{\lambda_k,\zeta_k}\stackrel{*}{\rightharpoonup}\mu_0$.
Since for every $k$ the ball $B_{\varepsilon_0}(w_0^i)$ contains exactly one atom
$w_{\lambda_k,\zeta_k}^i$ of $\mu_{\lambda_k,\zeta_k}$ with mass $c_{\lambda_k,\zeta_k}^i>0$,
and the balls are pairwise disjoint we obtain that
\begin{equation}
w_{\lambda_k,\zeta_k}^i\to w_0^i,
\qquad
c_{\lambda_k,\zeta_k}^i\to c_0^i,
\qquad i=1,\dots,N.
\end{equation}
\end{proof}

\begin{remark}
    Note that due to Theorem \ref{thm:uniqueness_hard_sol} that ensures uniqueness of solutions, the empirical measure $\mu_0$ is precisely the solution obtained by the application of the representer theorem \cite{ bartolucci2023understanding, boyer, bredies2020sparsity}. This justifies once more why the assumption $N \leq n$ is unavoidable. A similar observation can also be made regarding Theorem \ref{thm:int_rec}.
\end{remark}

\subsection{Rates of convergence and uniqueness of the coefficients}
In this section we derive quantitative rates for the convergences
\begin{align}
w_{\lambda,\zeta}^i \to w_0^i,
\qquad
c_{\lambda,\zeta}^i \to c_0^i,
\qquad i=1,\dots,N,
\end{align}
obtained in the previous section. We will show that, locally,
\begin{equation}
|c_{\lambda,\zeta}^i-c_0^i|+\|w_{\lambda,\zeta}^i-w_0^i\| \;=\; O\big(|\lambda|+\|\zeta\|\big),
\end{equation}
and in particular, along $N_{\alpha,\lambda_0}$, we get $O(\lambda)$.
%Given a collection of $w = \{w^1, \ldots, w^N\}$ where $w^i \in \mathbb{S}^{d-1}$ we write $\boldsymbol{\sigma}^i(w)=\left(\sigma\left(\left\langle w^i, x^1\right\rangle\right), \ldots, \sigma\left(\left\langle w^i, x^n\right\rangle\right)\right)\in \RR^n$, $i=1,\ldots,N$, where 
%\begin{equation}\label{pi}
%    \pi_j^i(w)=\left\{
%    \begin{aligned}
%        &1\quad \text{ if }\langle w^i,x^j\rangle>0,\\
%        &0\quad \text{ if }\langle w^i,x^j\rangle<0.
%    \end{aligned}
%    \right.
%\end{equation}
%In particular, using the previous definition \eqref{Kmat} of $\mathbf K(w_0)$, we have:
%  \begin{equation}
%(\mathbf K(w_0))^{\top}=\left[\begi%n{array}{cccc}
%\sigma(\langle w_0^1, x^1\rangle)  & \cdots & \sigma(\langle w_0^N, x^1\rangle)\\
%\vdots  & \cdots & \vdots \\
%\vdots  & \ddots & \vdots \\
%\sigma(\langle w_0^1, x^n\rangle) & \cdots & \sigma(\langle w_0^N, x^n\rangle)
%\end{array}\right]= \left[\begin{array}{cccc}
%\boldsymbol{\sigma}^1(w) & \cdots & \boldsymbol{\sigma}^N(w)\\
%\end{array}\right]\in \RR^{n\times N}.
%     \end{equation}
Note that 
\begin{equation}\label{K*mat}
\mathbf K(w_0) =   \left[\begin{array}{cccc}
K_*e^1(w_0^1)  & \cdots & K_*e^n(w_0^1)\\
\vdots  & \cdots & \vdots \\
\vdots  & \ddots & \vdots \\
K_*e^1(w_0^N) & \cdots & K_*e^n(w_0^N)
\end{array}\right]\in \RR^{N\times n},
     \end{equation}
     where $e^j\in \RR^n$ is the $j$-th vector of the canonical base of $\RR^n$. 
Given the operator $K_*' : \mathbb{R}^n \rightarrow C(\mathbb{S}^{d-1};\mathbb R^d)$ defined as  
     \begin{align}\label{eq:der_predual}
         K'_*(p^1,\ldots,p^n)(w) = \sum_{i=1}^n p^i \nabla_{\mathbb{S}^{d-1}} \sigma(\langle w, x^i\rangle),
     \end{align}
      where $\nabla_{\mathbb{S}^{d-1}}$ is the Riemannian gradient on the manifold $\mathbb{S}^{d-1}$,
     we can also define the $N(d-1)\times n$ matrix of first derivatives as
\begin{equation}\label{K'*mat}
\mathbf K'(w_0):=\left[\begin{array}{cccc}
\nabla_{\mathbb{S}^{d-1}}\sigma(\langle w_0^1, x^1\rangle)  & \cdots & \nabla_{\mathbb{S}^{d-1}}\sigma(\langle w_0^1, x^n\rangle)\\
\vdots  & \cdots & \vdots \\
\vdots  & \ddots & \vdots \\
\nabla_{\mathbb{S}^{d-1}}\sigma(\langle w_0^N, x^1\rangle) & \cdots & \nabla_{\mathbb{S}^{d-1}}\sigma(\langle w_0^N, x^n\rangle)
\end{array}\right]=  \left[\begin{array}{cccc}
K'_*e^1(w_0^1)  & \cdots & K'_*e^n(w_0^1)\\
\vdots  & \cdots & \vdots \\
\vdots  & \ddots & \vdots \\
K'_*e^1(w_0^N) & \cdots & K'_*e^n(w_0^N)
\end{array}\right],
     \end{equation}
     whose entries are tangent vectors in $\R^d$. In particular,  we note that
$K_*'p(w)\in T_w\mathbb{S}^{d-1}\subset \R^d$ for every $p\in\R^n$ and $w\in\mathbb{S}^{d-1}$. Therefore, the natural way to interpret \eqref{K'*mat} is as the linear map
\begin{equation}
\mathbf K'(w_0):\R^n \rightarrow \bigtimes_{i=1}^N T_{w_0^i}\mathbb{S}^{d-1}.
\end{equation}

\begin{theorem}[Interior Recovery]\label{thm:int_rec}
   Let $N \leq n$. Assume that $\mu_0=\sum_{i=1}^N c_0^i \sigma^i \delta_{w_0^i}$, where $c_0^i>0$, $\sigma^i\in \{-1,+1\}$, and $w_0^i \in R_{\pi^i}\cap \mathbb{S}^{d-1}$ for every $i=1, \ldots, N$, satisfies the $(LC)$. Moreover, suppose that $\{K\delta_{w_0^i}\}$ are linearly independent and
\begin{align}\label{eq:fullrank}
   \left[\begin{array}{c}
       \mathbf K(w_0) \\
        \mathbf K'(w_0) 
   \end{array}\right]
   \in \RR^{Nd\times n} \text{ has full rank } Nd.
   \end{align}
  Then, there exists $\alpha>0, \lambda_0>0$, such that, for all $(\lambda, \zeta) \in N_{\alpha, \lambda_0}$, the solution $\mu_{\lambda,\zeta}$ to \ref{eq:training} is unique and admits a unique representation of the form:
\begin{align}
\mu_{\lambda,\zeta}=\sum_{i=1}^N c_{\lambda,\zeta}^i \sigma^i\delta_{w_{\lambda,\zeta}^i},
\end{align}
where $w_{\lambda,\zeta}^i \in R_{\pi^i}\cap \mathbb{S}^{d-1}$ such that $\eta_{\lambda,\zeta} (w_{\lambda,\zeta}^i)=\sigma^i, c_{\lambda,\zeta}^i>0$. Moreover, for every $i=1, \ldots, N$ and all $(\lambda, \zeta) \in N_{\alpha, \lambda_0}$, it holds: 
\begin{equation}\label{eq:Lipschitz_rate}
    \begin{aligned}
\left|c_{\lambda,\zeta}^i-c_0^i\right|&=O(\lambda), \\
\left\|w_{\lambda,\zeta}^i-w_0^i\right\|&=O(\lambda).
    \end{aligned}
\end{equation}
\end{theorem}
\begin{proof} 
Since the (LC) holds for $\mu_0$ and the measurements are linearly independent, we can apply Proposition \ref{unique_loc}. Once again we will carry out the proof considering $\sigma^i=1$ for simplicity. Therefore, there exist $\alpha>0, \lambda_0>0$ such that, for all $(\lambda, \zeta) \in N_{\alpha, \lambda_0}$, any solution $\mu_{\lambda,\zeta}$ is composed of exactly $N$ Dirac deltas, i.e.
$$
\mu_{\lambda,\zeta}=\sum_{i=1}^N c_{\lambda,\zeta}^i \delta_{w_{\lambda,\zeta}^i},
$$
where $c_{\lambda,\zeta}^i>0$ and $w_{\lambda,\zeta}^i \in R_{\pi^i} \cap \mathbb{S}^{d-1}$ such that $\eta_{\lambda,\zeta}( w_{\lambda,\zeta}^i)=1$ for every $i=1, \ldots, N$. In Proposition \ref{unique_loc} we also obtained the uniqueness of the locations $w_{\lambda,\zeta}^i$ inside each region $R_{\pi^i}$.
To conclude uniqueness of $\mu_{\lambda,\zeta}$ it remains to show the uniqueness of the coefficients $c_{\lambda,\zeta}^i$,
and to obtain the quantitative estimate \eqref{eq:Lipschitz_rate}.
We work in a neighborhood where each $w_0^i$ is interior to its activation region, so $\pi^i$ is fixed.
In particular, we choose $\varepsilon>0$ such that $B_\varepsilon(w_0^i)\cap\mathbb{S}^{d-1}\subset R_{\pi^i}\cap\mathbb{S}^{d-1}$ for all $i$. 
Now, we define \begin{equation}\label{eq:R_def}
R(c,w,\zeta):=K\sum_{i=1}^N c^i \delta_{w^i}-y_0-\zeta\in \R^n,
\qquad R(c_0,w_0,0)=0,
\end{equation}
for $c\in\R^N$ and $w=(w^1,\ldots,w^N)\in \prod_{i=1}^N\big(B_\varepsilon(w_0^i)\cap\mathbb{S}^{d-1}\big)$. Since $\mu_{\lambda,\zeta}$ satisfies the optimality conditions, $\eta_{\lambda,\zeta}$ reaches its maximum at $w_{\lambda,\zeta}^j$,
i.e.\ $K_*p_{\lambda,\zeta}(w_{\lambda,\zeta}^j)=\eta_{\lambda,\zeta}(w_{\lambda,\zeta}^j)=1$. Moreover, since $w_{\lambda, \zeta}^j$ lies in the interior of $R_{\pi^j}$, the map $\eta_{\lambda, \zeta}$ is $C^1$ near $w_{\lambda, \zeta}^j$, hence
$\nabla_{\mathbb{S}^{d-1}}\eta_{\lambda,\zeta}(w_{\lambda,\zeta}^j)=0$. Recalling
$-\lambda p_{\lambda,\zeta}=K\mu_{\lambda,\zeta}-y_0-\zeta$ and $\eta_{\lambda,\zeta}=K_*p_{\lambda,\zeta}$, we obtain
\begin{equation}\label{eq:implicit_eq_values}
K_*\big(R(c_{\lambda,\zeta},w_{\lambda,\zeta},\zeta)\big)(w_{\lambda,\zeta}^j)+\lambda=0,
\qquad j=1,\ldots,N,
\end{equation}
and
\begin{equation}\label{eq:implicit_eq_grad}
K_*'\big(R(c_{\lambda,\zeta},w_{\lambda,\zeta},\zeta)\big)(w_{\lambda,\zeta}^j)=0
\quad\text{in }T_{w_{\lambda,\zeta}^j}\mathbb{S}^{d-1},
\qquad j=1,\ldots,N.
\end{equation}
Moreover, for $(\lambda,\zeta)=(0,0)$ we have $K\mu_0=y_0$, hence the same equations hold at $(c_0,w_0)$. To apply the implicit function theorem, we introduce local charts around $w_0^{j}$ and express the second equation
\eqref{eq:implicit_eq_grad} in tangent coordinates.

For each $j$, we choose a smooth chart $\varphi_j:U_j\subset\R^{d-1}\to \mathbb{S}^{d-1}$ such that
$\varphi_j(0)=w_0^j$ and $\varphi_j(U_j)\subset B_\varepsilon(w_0^j)\cap\mathbb{S}^{d-1}$, and write
$u=(u^1,\ldots,u^N)\in U:=U_1\times\cdots\times U_N\subset\R^{N(d-1)}$ such that 
$w(u):=\big(\varphi_1(u^1),\ldots,\varphi_N(u^N)\big)$.
We can then choose, for each $j$, a $C^1$ family of linear isometries
\begin{equation}\label{eq:isometries}
\Pi_j(u^j):T_{\varphi_j(u^j)}\mathbb{S}^{d-1}\rightarrow \R^{d-1},
\qquad \Pi_j(u^j)\ \text{is an isometry for every }u^j\in U_j.
\end{equation}
Then, we can define $F=(F^1,\ldots,F^N):(\R^N\times U)\times N_{\alpha,\lambda_0}\to \R^{Nd}$ by
\begin{equation}\label{eq:implicit_function_coord}
F^j(c,u;\lambda,\zeta):=
\left(\begin{aligned}
K_*\big(R(c,w(u),\zeta)\big)\big(\varphi_j(u^j)\big)+\lambda\\[2pt]
\Pi_j(u^j)\,K_*'\big(R(c,w(u),\zeta)\big)\big(\varphi_j(u^j)\big)
\end{aligned}\right)\in\R^{1+(d-1)}=\R^d.
\end{equation}
Since the activation pattern is fixed on $\varphi_j(U_j)\subset R_{\pi^j}$ and the maps $\varphi_j,\Pi_j$ are $C^1$,
the map $(c,u,\lambda,\zeta)\mapsto F(c,u;\lambda,\zeta)$ is $C^1$ near $(c_0,0;0,0)$.
Moreover, \eqref{eq:implicit_eq_values}--\eqref{eq:implicit_eq_grad} are equivalent to
\begin{equation}\label{eq:F_zero}
F(c_{\lambda,\zeta},u_{\lambda,\zeta};\lambda,\zeta)=0,
\end{equation}
where $u_{\lambda,\zeta}$ is defined by $w_{\lambda,\zeta}^j=\varphi_j(u_{\lambda,\zeta}^j)$ (for $(\lambda,\zeta)$ small). It also holds that $F(c_0,0;0,0)=0$. Now, we differentiate $F$ with respect to $(c,u)$ and we check the invertibility of $D_{(c,u)}F(c_0,0;0,0)$. First, note that for fixed $w=(w^1,\ldots,w^N)$, we have
\begin{equation}
p\mapsto \mathbf K(w)\,p
=\big((K_*p)(w^1),\ldots,(K_*p)(w^N)\big)\in\R^N,
\end{equation}
and, in tangent coordinates defined by $\Pi_j(u^j)$,
\begin{equation}
p\mapsto \mathbf K'_T(w(u))\,p
:=\big(\Pi_1(u^1)K_*'p(\varphi_1(u^1)),\ldots,\Pi_N(u^N)K_*'p(\varphi_N(u^N))\big)\in\R^{N(d-1)}.
\end{equation}
At $u=0$ this gives the coordinate representation of $\mathbf K'(w_0)$, which we denote by $\mathbf K'_T(w_0)$. With this notation, we can define 
\begin{equation}
\mathbf A_T(w(u)):=
\left[\begin{array}{c}
\mathbf K(w(u))\\
\mathbf K'_T(w(u))
\end{array}\right]\in\R^{Nd\times n}.
\end{equation}
If we differentiate in $(c,u)$, we get an $Nd \times Nd$ matrix:
\begin{equation}\label{eq:Diff_F}
    D_{(c, u)}F =D_u\mathbf A_T(w(u))R(c, w(u), \zeta)+\mathbf A_T(w(u)) D_{(c, u)} R(c,w(u),\zeta),
\end{equation}
where the first term vanishes when evaluated at $(c_0,0;0,0)$, that is
\begin{equation}\label{eq:Diff_F_0}
    D_{(c, u)} F\left(c_0, 0 ; 0,0\right)=\mathbf A_T(w_0) D_{(c, u)} R\left(c_0, w_0, 0\right).
\end{equation}
Therefore, the Jacobian depends only on the first order variation of $R$. Next, differentiating
$R(c,w(u),\zeta)$ at $(c_0,0,0)$ yields the block form
\begin{equation}\label{eq:DR}
D_{(c,u)}R(c_0,w_0,0)=
\Big[\mathbf K(w_0)^\top\ \quad\ \mathbf K'_T(w_0)^\top\operatorname{diag}^{\,d-1}(c_0)\,M\Big],
\end{equation}
where $\operatorname{diag}^{\,d-1}(c_0)\in\R^{N(d-1)\times N(d-1)}$ is the block-diagonal matrix with $(d-1)$ copies of each $c_0^i$,
and $M\in\R^{N(d-1)\times N(d-1)}$ is the block-diagonal matrix
\begin{equation}
M:=\operatorname{diag}(M_1,\ldots,M_N),\qquad
M_j:=\Pi_j(0)\circ D\varphi_j(0)\in\R^{(d-1)\times(d-1)}.
\end{equation}
Since each $\varphi_j$ is a local chart and each $\Pi_j(0)$ is an isometry, every $M_j$ (hence $M$) is invertible.
Substituting \eqref{eq:DR} into \eqref{eq:Diff_F_0} we obtain the following factorization
\begin{equation}\label{eq:jac_factor}
D_{(c,u)}F(c_0,0;0,0)
=
\mathbf A_T(w_0)\,\mathbf A_T(w_0)^{\top}
\left[\begin{array}{cc}
I_N & 0\\
0 & \operatorname{diag}^{\,d-1}(c_0)\,M
\end{array}\right].
\end{equation}
Fix the linear isometries $\Pi_j(0):T_{w_0^j}\mathbb{S}^{d-1}\to\R^{d-1}$ for $j=1,\ldots,N$, and consider the following linear map
\begin{equation}
J:\R^N\bigtimes_{j=1}^N T_{w_0^j}\mathbb{S}^{d-1}\rightarrow \R^N\times \R^{N(d-1)}
\qquad
J(a;v_1,\ldots,v_N)=(a;\Pi_1(0)v_1,\ldots,\Pi_N(0)v_N).
\end{equation}
Then, $J$ is an isomorphism and the matrix $\mathbf A_T(w_0)$ is precisely the coordinate representation of the linear map
$J\circ \left[\begin{smallmatrix}\mathbf K(w_0)\\ \mathbf K'(w_0)\end{smallmatrix}\right]$. Since after the composition with an isomorphism the rank does not change,
\begin{equation}
\operatorname{rank}\!\left(\mathbf A_T(w_0)\right)
=\operatorname{rank}\!\left(\left[\begin{array}{c}\mathbf K(w_0)\\ \mathbf K'(w_0)\end{array}\right]\right).
\end{equation}
In particular, by \eqref{eq:fullrank} we have $\operatorname{rank}(\mathbf A_T(w_0))=Nd$ (full row rank), hence
$\mathbf A_T(w_0)\mathbf A_T(w_0)^\top$ is symmetric positive definite. Moreover, since $c_0^i>0$ and $M$ is invertible,
the block diagonal matrix in \eqref{eq:jac_factor} is invertible. Therefore, $D_{(c,u)}F(c_0,0;0,0)$ is invertible. The implicit function theorem applies. Hence, there exist neighborhoods $B_r(c_0,0)\subset \R^N\times U$ and $B_s((0,0))\subset \R\times \R^n$,
and a unique $C^1$ map
\begin{equation}
g:B_s((0,0))\to B_r(c_0,0),\qquad (\lambda,\zeta)\mapsto (c(\lambda,\zeta),u(\lambda,\zeta)),
\end{equation}
such that $g(0,0)=(c_0,0)$ and $F(g(\lambda,\zeta);\lambda,\zeta)=0$ for all $(\lambda,\zeta)\in B_s((0,0))$. In particular, since $g$ is $C^1$ at $(0,0)$, there exists $C>0$ such that, for $(\lambda,\zeta)$ sufficiently small,
\begin{equation}\label{eq:g_lip_sub}
\|(c(\lambda,\zeta),u(\lambda,\zeta))-(c_0,0)\|
\leq C\big(|\lambda|+\|\zeta\|\big).
\end{equation}
Mapping back to the sphere by $w^i(\lambda,\zeta):=\varphi_i(u^i(\lambda,\zeta))$ and using that $\varphi_i$ is $C^1$,
\eqref{eq:g_lip_sub} implies 
\begin{equation}\label{eq:g_lip}
\left|c_{\lambda,\zeta}^i-c_0^i\right|+\left\|w_{\lambda,\zeta}^i-w_0^i\right\|
\;\leq\; C\big(|\lambda|+\|\zeta\|\big).
\end{equation}
Finally, by Proposition \ref{unique_loc} the locations $w_{\lambda,\zeta}^i$ are uniquely determined inside each region $R_{\pi^i}$
and satisfy $w_{\lambda,\zeta}^i\to w_0^i$, together with $c_{\lambda,\zeta}^i\to c_0^i$. Therefore, for $(\lambda,\zeta)$ small enough, we have $(c_{\lambda,\zeta},u_{\lambda,\zeta})\in B_r(c_0,0)$.
Since $F(c_{\lambda,\zeta},u_{\lambda,\zeta};\lambda,\zeta)=0$ by \eqref{eq:F_zero}, uniqueness in the IFT implies
\begin{equation}
(c_{\lambda,\zeta},u_{\lambda,\zeta})=g(\lambda,\zeta),
\end{equation}
which proves that $c_{\lambda,\zeta}$ is unique and therefore $\mu_{\lambda,\zeta}$ is unique.
The rates \eqref{eq:Lipschitz_rate} follow from \eqref{eq:g_lip}, since on $N_{\alpha,\lambda_0}$ it holds $\|\zeta\|\leq \alpha\lambda$.
\end{proof}

%
%\begin{remark}
%    Note that if $N \geq n$, then $\mathbf K(w_0)$ full rank implies \eqref{eq:fullrank}. In particular, if $N=n$ then the linear independence of $\{K\delta_{w^i_0}\}$ implies \eqref{eq:fullrank}. Moreover if $n \geq N(d+1)$, then one recovers the setting in \cite{Peyre} where $n = \infty$ and \eqref{eq:fullrank} immediately implies the linear independence of $\{K\delta_{w^i_0}\}$.
%\end{remark}
\begin{remark}\label{rem:ND_not_needed}
Note that the condition \eqref{eq:fullrank} forces $n\geq Nd$. Moreover, the linear independence of $\{K\delta_{w_0^i}\}_{i=1}^N$ is not implied by \eqref{eq:fullrank} in general; it is assumed here in order to invoke Proposition~\ref{unique_loc} and obtain the localization of the reconstructed Dirac deltas. We note that \eqref{eq:fullrank} can be seen as a finite-dimensional analogue of the condition used in the infinite-dimensional measurement setting of \cite{Peyre}, where one works with an observation operator having $L^2$ measurements (heuristically $n=+\infty$).
Regarding the non-degeneracy condition we remark that we do not impose (ND), since the locations lie in the interior of the dual regions; in this case, (LC) is sufficient to apply Proposition~\ref{unique_loc}.

Finally, we remark that a similar argument would provide a linear decay of both the location and the coefficient errors for fixed $\lambda_0$ and a vanishing sequence of noise levels $\zeta$. In this case, the assumptions concerning the linear independence of the measurements and \eqref{eq:fullrank} should clearly be understood with respect to the zero-noise reconstruction with $\lambda =\lambda_0$.
\end{remark}

\section{Numerical experiments}

This section shows simple numerical simulations confirming the theoretical findings of this paper. For simplicity, we focus on the two-dimensional setting ($d=2$), where the weight domain is the unit circle $\mathbb S^1$. The experiments are designed to display the dual-region partition, the behaviour of the dual certificate $\eta_{\lambda,\zeta}$ on $\mathbb S^1$, and of the optimal solution $\mu_{\lambda,\zeta}$ depending on noise level and regularization parameter. In particular, we show the sparsity of the recovered measure (Theorem \ref{cor:exactreconstruction}), the relationship between the number of atoms and the regularization parameter (Lemma \ref{lem:crudebound}), and the linear convergence rates for the weights and locations under label noise (Theorem \ref{thm:int_rec}).
These examples are representative, since the aim here is not to optimize predictive performance but to isolate and visualize the results discussed in the previous sections.

\subsection{Model and experimental setup}
Given inputs $x_1,\dots,x_n\in\mathbb R^2$ and targets $y\in\mathbb R^n$, we consider the operator
\begin{equation}\label{eq:model_measure}
(K\mu)_j =\int_{\mathbb S^1}\sigma(\langle w,x^j\rangle)\,d\mu(w),
\qquad \sigma(t)={\rm ReLU}(t),
\end{equation}
and the TV-regularized empirical risk problem:
\begin{equation}\label{eq:tv_problem}
\min_{\mu\in M(\mathbb S^1)}\;
\frac12\sum_{j=1}^n((K\mu)_j-y_j)^2 +\lambda\|\mu\|_{\mathrm{TV}},
\qquad \lambda>0.
\end{equation}
Numerically, we approximate $\mu$ by an empirical measure with $m$ Dirac deltas:
\begin{equation}\label{eq:particle}
\mu =\sum_{i=1}^m c_i\,\delta_{w_i},
\qquad w_i\in\mathbb S^1,\quad c_i\in\mathbb R,
\end{equation}
which yields the finite-dimensional objective
\begin{equation}\label{eq:finite_objective}
\min_{\{(c_i,w_i)\}_{i=1}^m} \frac12\sum_{j=1}^n\Big(\sum_{i=1}^m c_i\,\sigma(\langle w_i,x_j\rangle)-y_j\Big)^2
+ \lambda\sum_{i=1}^m |c_i|.
\end{equation}
We enforce the constraint $w_i\in\mathbb S^1$ by parametrizing $w_i=u_i/\|u_i\|$ with unconstrained $u_i\in\mathbb R^2$ and optimizing over $(c_i,u_i)$. 

\subsection{Simulations}

We train the model by optimizing \eqref{eq:finite_objective} with ADAM using $m = 10^5$ and random initialization for $c$ and $u$. 
In all the experiments reported in Figures~\ref{fig:gated_vec_field}--\ref{fig:stability_2}, we use a fixed set of $n=5$ data points in $\mathbb R^2$:
\begin{equation}\label{eq:data_points}
x_1=(0.2,-0.1),\quad
x_2=(1.0,0.3),\quad
x_3=(1.0,0.0),\quad
x_4=(-0.4,0.9),\quad
x_5=(0.5,0.5)
\end{equation}
and targets
\begin{equation}\label{eq:targets}
y=(0.8,\,-0.1,\,0.3,\,-1.2,\,1.0).
\end{equation}
We recall that in the first three experiments the labels are not corrupted by noise ($\zeta=0$), while in the final experiment we vary the noise level to verify the convergence rates of Theorem \ref{thm:int_rec}.
The data points $x_i$ are \emph{non-collinear} and span multiple directions, so the hyperplane constraints $\langle w,x_j\rangle=0$ generate a nontrivial arrangement on $\mathbb S^1$ with multiple dual regions (see Figure \ref{fig:gated_vec_field}). 
As a regularization parameter, we choose $\lambda = 0.03$. Every $200$ iterations we prune the sequence by removing the atoms corresponding to coefficients with magnitude under the threshold $\tau_{prune} = 10^{-2}$.
The result is represented in Figure \ref{fig:gated_vec_field} and it corresponds to the solution 
\begin{align}\label{eq:fin}
    \mu_{\lambda,\zeta} = \sum_{i=1}^3 c_*^i \delta_{w_*^i},
\end{align}
where 
\begin{align*} 
w_*^1 = (-0.163 , 0.987),\, w_*^2 = (0.287 , -0.958), w_*^3 = (-0.708, 0.706),
\end{align*}
\begin{align*}
 c_*^1 = 1.312, c_*^2 = 1.256, c_*^3 = -2.577.
\end{align*}

\begin{figure}[h!]
    \centering
\includegraphics[width=0.5\textwidth]{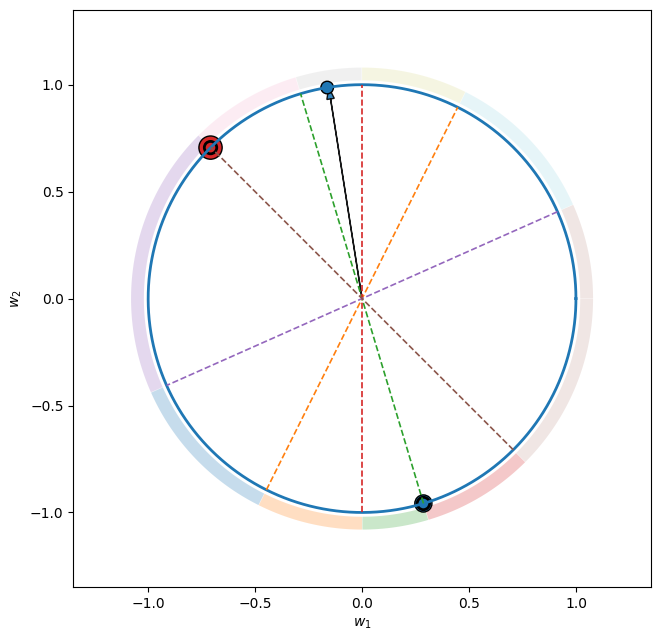}
     \caption{\textbf{Dual region decomposition of $\mathbb S^1$ and optimal solution $\mu_{\lambda,\zeta}$.}
    The dashed diameters correspond to the boundaries $\langle w,x_j\rangle=0$ induced by the five inputs~\eqref{eq:data_points}, partitioning $\mathbb S^1$ into sectors on which $(\mathbf 1\{\langle w,x_j\rangle>0\})_{j=1}^5$ is constant.
     We display the recovered atoms $\{w_i\}_{i=1}^3$  of $\mu_{\lambda,\zeta}$. The arrow corresponds to the vector $\frac{\tilde z_{\lambda,\zeta}}{\|\tilde z_{\lambda,\zeta}\|}$ as in \eqref{eq:gated_vector_2} with $\tilde z_{\lambda,\zeta}$ defined in \eqref{eq:gated_vector_1}. As predicted by Lemma \ref{unique_c.p.soft}, this vector aligns with the location of the Dirac delta of the solution $\mu_{\lambda,\zeta}$ belonging to the interior of a dual region.}
    \label{fig:gated_vec_field}
\end{figure}
We also study the behaviour of the dual certificate $\eta_{\lambda,\zeta} \in C(\mathbb{S}^{d-1})$ corresponding to the computed solution $\mu_{\lambda,\zeta}$. In particular, we show that it reaches the maximum $\eta_{\lambda,\zeta}= 1$ at the location of the positive Dirac deltas of the optimal solution; similarly, it reaches the minimum $\eta_{\lambda,\zeta} = -1$ at the location of the negative Dirac deltas. Such behaviour is a clear indication that the optimization procedure has reached a minimizer of the problem.

Moreover, we note that the dual certificate aligns with the unique (positive) Dirac delta of $\mu_{\lambda, \zeta}$ lying in the interior of a dual region, which we denote by $w_* \in R_\pi$.
As described in Lemma \ref{unique_c.p.soft}, it holds that, for $p_{\lambda,\zeta} = -\frac{1}{\lambda}(K\mu_{\lambda,\zeta} - y_{\zeta})$ (see \eqref{oc_soft}) and $\tilde z_\lambda = \sum_{i} p_\lambda^i \pi_i x^i$ (see  \eqref{eq:gated_vector_1}),
\begin{align}
    \eta_{\lambda,\zeta} = {\rm sign}(w_*)\frac{\tilde z_{\lambda,\zeta}}{\|\tilde z_{\lambda,\zeta}\|}.
\end{align}

\begin{figure}[!t]
    \centering
\includegraphics[width=0.5\textwidth]{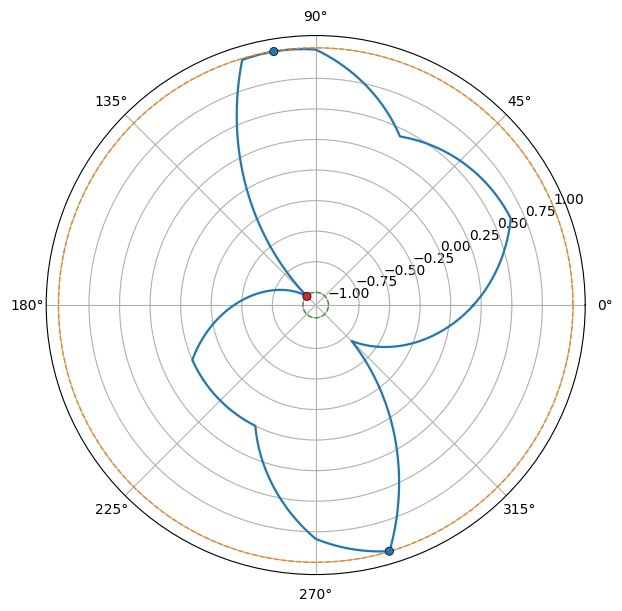}
    \caption{\textbf{Dual certificate $\eta_{\lambda,\zeta}$ on $\mathbb S^1$.}
    The dashed circles (orange and green) show the locations where $\eta_{\lambda,\zeta}=\pm 1$. As expected by the optimality conditions, the dual certificate is always between $-1$ and $+1$.
    The points where $\eta_\lambda = 1$ (resp. $\eta_\lambda = -1$) correspond exactly to the location of positive (resp. negative) Dirac deltas in the TV-regularized solution \eqref{eq:fin}.}
    \label{fig:dual_cert}
\end{figure}

In Figure \ref{fig:number_atoms}, while keeping fixed the data and labels, we modify the value of the regularization parameter to observe how the number of reconstructed Dirac deltas varies in the solution $\mu_{\lambda,\zeta}$. We note that when the regularization parameter decreases, the number of Dirac deltas increases, while staying below the geometric upper bound given by Lemma \ref{lem:crudebound}, which in this case corresponds to $10$.

\begin{figure}[!htbp]
    \centering
\includegraphics[width=0.6\textwidth]{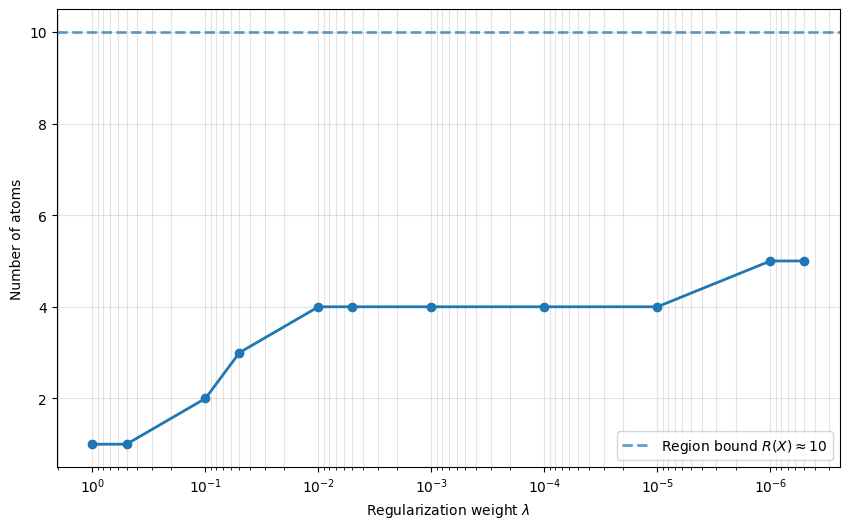}
   \caption{\textbf{Size of the support depending on the regularization parameter.}
    For $\lambda$ in the interval  ($\lambda_{\max}=1$, $\lambda_{\min}=5\cdot10^{-7}$, $11$ values), we compute the cardinality of the support of the solution. As expected, the number of Dirac deltas increases with the decrease of the regularization parameter $\lambda$. The maximum number of Dirac deltas is, however, still below the theoretical bound determined by the number of regions: $R(X) = 10$.}
    \label{fig:number_atoms}
\end{figure}

\medskip

Finally, in the last experiment, we verify the decay of the coefficients and locations of the Dirac deltas in the optimal solution $\mu_{\lambda,\zeta}$ predicted by Theorem \ref{thm:int_rec} (see also Remark \ref{rem:ND_not_needed}).  We again work in dimension $d=2$ and we consider $n=5$ data points in $\mathbb R^2$:
\begin{equation}\label{eq:data_points_2}
x_1=(0.2,-0.1),\quad
x_2=(1.0,-2),\quad
x_3=(1.0,0.2),\quad
x_4=(-0.4,-0.2),\quad
x_5=(0.5,0.5)
\end{equation}
and targets
\begin{equation}\label{eq:targets_2}
y=(0.8,\,-0.1,\,0.3,\,-1.2,\,1.0).
\end{equation}
We fix $\lambda=0.2$ and we compute the solution $\mu_{\lambda,\zeta} = \sum_{i=1}^2 c_*^i \delta_{w_*^i}$ with no noise. The result is depicted in Figure \ref{fig:stability_1}, where we observe that both $w_*^i$ lie in the interior of a dual region.

\begin{figure}[!t] 
\begin{minipage}{0.49\textwidth}
    \centering
\includegraphics[width=0.9\textwidth]{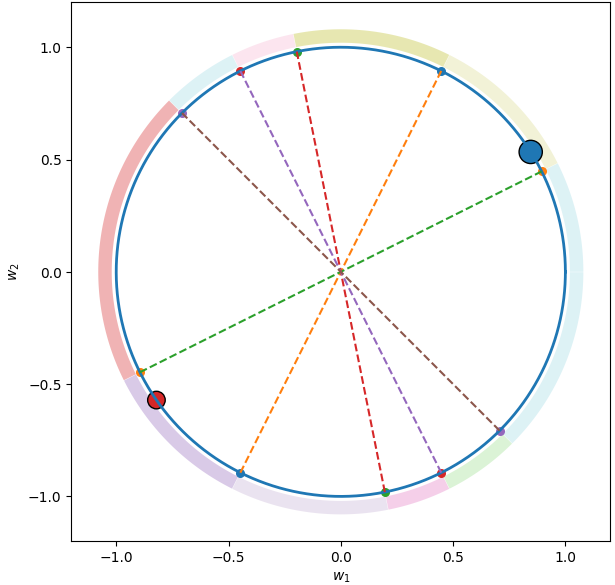} 
\end{minipage}
\begin{minipage}
{0.49\textwidth}
  \centering
\includegraphics[width=0.9\textwidth]{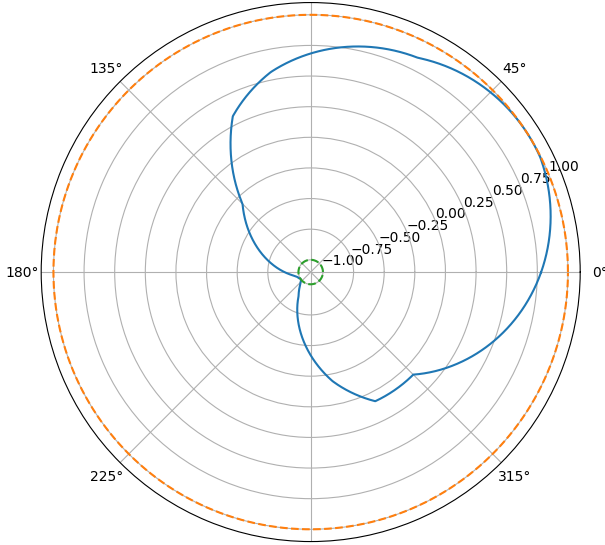}    
\end{minipage}
\caption{\textbf{Unperturbed solution and dual certificate for $\lambda = 0.2$.}
\emph{Left:} Dual region decomposition of $\mathbb{S}^1$ and recovered 
atoms $\{w_*^i\}_{i=1}^2$ of the solution 
$\mu_{\lambda,\zeta} = \sum_{i=1}^2 c_*^i\,\delta_{w_*^i}$ 
with $\zeta = 0$. Both atoms lie in the interior of a dual region.
\emph{Right:} Corresponding dual certificate $\eta_{\lambda,\zeta}$ 
on $\mathbb{S}^1$. As in Figure~\ref{fig:dual_cert}, it satisfies 
$|\eta_{\lambda,\zeta}| \leq 1$ everywhere and attains $\pm 1$ 
exactly at the support of $\mu_{\lambda,\zeta}$.}
\label{fig:stability_1}
 \end{figure}

While keeping $\lambda$ fixed, we perturb $y$ by increasingly large noise $\zeta$, ranging from $5\cdot 10^{-4}$ to $2.4 \cdot 10^{-3}$. We compute the optimal solution $\mu_{\lambda,\zeta} = \sum_{i=1}^2 c_\zeta^i \delta_{w_\zeta^i}$ for all values of $\zeta$ and in Figure \ref{fig:stability_2} we plot the quantities
\begin{align*}
L^i_c(\zeta) = |c_{\lambda,\zeta}^i -  c_*^i| \quad \text{and} \quad L^i_w(\zeta) = \|w_{\lambda,\zeta}^i -  w_*^i\| \quad \text{for} \quad i=1,2.
\end{align*}
As predicted by Theorem \ref{thm:int_rec}, we observe that these quantities decay linearly in $\|\zeta\|$, consistently with the bound $O(\lambda)$ along $N_{\alpha, \lambda_0}$ where $\|\zeta\| \leq \alpha \lambda$. In particular, we have $L^i_c(\zeta) \leq C_c^i\|\zeta\|$ and $L^i_w(\zeta) \leq C_w^i\|\zeta\|$ for constants $C_c^i$ and $C_w^i$ with $i=1,2$. 

One can also verify that the assumptions of Theorem~\ref{thm:int_rec} are indeed 
satisfied. Since both atoms $w_*^1, w_*^2$ 
lie in the interior of their respective dual regions 
$R_{\pi_1}, R_{\pi_2}$, the non-degeneracy condition~(ND) is not 
required (cf.\ Remark~\ref{rem:ND_not_needed}). The localization 
condition~(LC) is verified by inspecting the dual certificate 
$\eta_{\lambda,\zeta}$ in Figure~\ref{fig:stability_1} (right), where 
we observe that $|\eta_{\lambda,\zeta}(w)| < 1$ for all 
$w \in \mathbb{S}^1 \setminus \{w_*^1, w_*^2\}$. To verify the linear independence of 
$\{K\delta_{w_*^1}, K\delta_{w_*^2}\}$ we apply 
Lemma~\ref{lem:lin_ind_meas}. The activation patterns of the two 
atoms are $\pi_1 = (1,0,1,0,1)$ and $\pi_2 = (1,1,0,1,0)$. 
The $2 \times 2$ minor of $\Pi$ corresponding to columns 
$1$ and $2$ equals 
$\begin{psmallmatrix} 1 & 0 \\ 1 & 1 \end{psmallmatrix}$, 
which satisfies 
$\operatorname{perm}(\Pi_N) = |\det(\Pi_N)| = 1$, so 
conditions~(i) and~(ii) of Lemma~\ref{lem:lin_ind_meas} are fulfilled. Finally, we verify condition~\eqref{eq:fullrank}. Since $d = 2$, 
the tangent space $T_{w_*^i}\mathbb{S}^1$ is one-dimensional, 
spanned by $(-w_{*,2}^i,\, w_{*,1}^i)$. 
%In 
%this basis, the entries of the derivative matrix 
%$\mathbf{K}'(w_*) \in \mathbb{R}^{2 \times 5}$ read
%\[
%[\mathbf{K}'(w_*)]_{ij} = \pi_j^i\,\langle x^j, {w_*^i}^\perp \rangle, 
%\qquad i = 1,2,\quad j = 1,\ldots,5.
%\]
We numerically evaluate the matrix
\[
\begin{bmatrix} \mathbf{K}(w_*) \\ \mathbf{K}'(w_*) \end{bmatrix} 
\in \mathbb{R}^{4 \times 5}
\]
and obtain singular values approximately 
$2.243$, $1.204$, $0.472$, $0.365$, confirming that its rank is $4 = Nd$. 
Hence, all hypotheses of Theorem~\ref{thm:int_rec} are fulfilled.

%As a consistency check we also verify that the Assumptions of Theorem \ref{thm:int_rec} are satisfied. In particular 
%the matrix 
%\[\left[\begin{array}{c}
  %     \mathbf K(w_*) \\
 %       \mathbf K'(w_*) 
 %  \end{array}\right] \in \]
%has rank equal to  with singular values.

\begin{figure}[h!]
\begin{minipage}{0.49\textwidth}
    \centering    \includegraphics[width=0.9\textwidth]{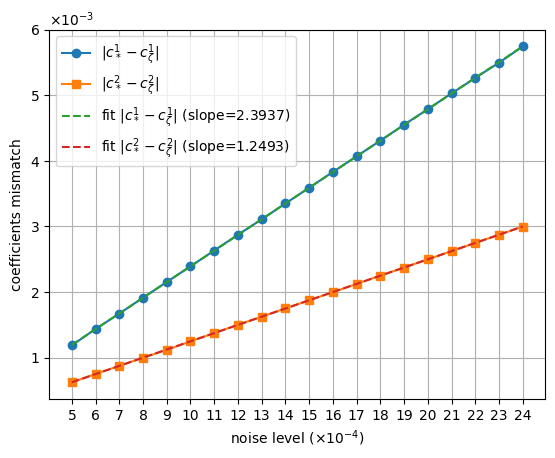} 
\end{minipage}
\begin{minipage}
{0.49\textwidth}
  \centering
\includegraphics[width=0.9\textwidth]{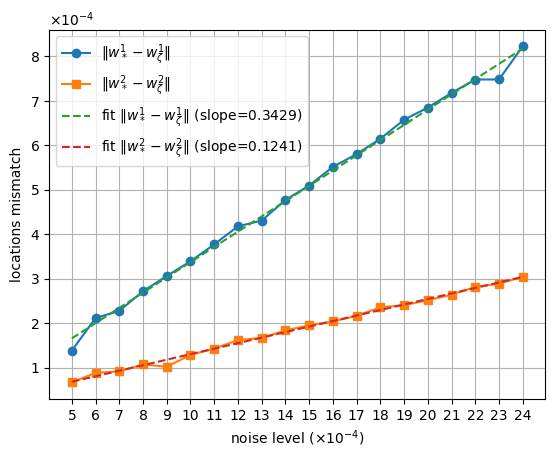}    
\end{minipage}
\caption{\textbf{Linear decay of coefficients and locations under 
label noise.}
\emph{Left:} Coefficient mismatch 
$L_c^i(\zeta) = |c_{\lambda,\zeta}^i - c_*^i|$ as a function of 
the noise level $\|\zeta\|$ for $i=1,2$, together with linear fits.
\emph{Right:} Location mismatch 
$L_w^i(\zeta) = \|w_{\lambda,\zeta}^i - w_*^i\|$ for $i=1,2$.
Both quantities exhibit the linear dependence on $\|\zeta\|$ 
predicted by Theorem \ref{thm:int_rec}.}
\label{fig:stability_2}
 \end{figure}

\bibliographystyle{plain}
		\bibliography{BibESRRNN}

%\appendix
%\large{\section*{Appendix}}
%\renewcommand{\thesection}{A} 

%\normalsize
%\subsection{Uniqueness of the weights}

\end{document}